\documentclass[11pt, letterpaper]{article}

\usepackage{times}
\usepackage{wrapfig}
\usepackage{sectsty}
\sectionfont{\large}
\subsectionfont{\normalsize}
\subsubsectionfont{\small}
\usepackage{listings}
\usepackage{caption}
\usepackage{subcaption}

\usepackage{amsfonts}
\usepackage{microtype}
\usepackage{graphicx}

\usepackage{booktabs}

\usepackage{multicol}
\usepackage{algorithm}
\usepackage{algpseudocode}
\usepackage{enumitem}

\usepackage{setspace}
\usepackage{natbib}

\usepackage{amsmath}
\usepackage{amssymb}
\usepackage{mathtools}
\usepackage{amsthm}

\usepackage[colorlinks,
            linkcolor=red,
            anchorcolor=blue,
            citecolor=blue
            ]{hyperref}
\usepackage[capitalize,noabbrev]{cleveref}

\DeclareFontFamily{U}{mathb}{}
\DeclareFontShape{U}{mathb}{m}{n}{
      <-6> mathb5
      <6-7> mathb6
      <7-8> mathb7
      <8-9> mathb8
      <9-10> mathb9
      <10-12> mathb10
      <12-> mathb12
}{}
\DeclareSymbolFont{mathb}{U}{mathb}{m}{n}

\DeclareMathSymbol{\drsh}{3}{mathb}{"EB}

\theoremstyle{plain}
\newtheorem{theorem}{Theorem}[section]
\newtheorem{proposition}[theorem]{Proposition}

\theoremstyle{definition}
\newtheorem{definition}[theorem]{Definition}
\newtheorem{assumption}[theorem]{Assumption}

\usepackage{smile}
\usepackage{makecell}

\usepackage{amsmath,amsfonts,bm}

\newcommand*{\dif}{\mathop{}\!\mathrm{d}}

\DeclareMathOperator\supp{supp}
\DeclareMathOperator\diam{diam}

\newcommand{\interior}[1]{%
  {\kern0pt#1}^{\mathrm{o}}%
}

\def\1{\bm{1}}

\def\vzero{{\bm{0}}}
\def\vone{{\bm{1}}}

\def\va{{\bm{a}}}
\def\vb{{\bm{b}}}

\def\vd{{\bm{d}}}

\def\vh{{\bm{h}}}

\def\vv{{\bm{v}}}

\def\vx{{\bm{x}}}
\def\vy{{\bm{y}}}
\def\vz{{\bm{z}}}

\def\eva{{a}}

\def\evc{{c}}
\def\evd{{d}}

\def\evh{{h}}

\def\evv{{v}}

\def\evy{{y}}

\def\mLambda{{\bm{\Lambda}}}

\def\mLambda{{\bm{\Lambda}}}
\def\mDelta{{\bm{\Delta}}}

\DeclareMathAlphabet{\mathsfit}{\encodingdefault}{\sfdefault}{m}{sl}
\SetMathAlphabet{\mathsfit}{bold}{\encodingdefault}{\sfdefault}{bx}{n}

\def\gA{{\mathcal{A}}}

\def\gE{{\mathcal{E}}}

\def\gG{{\mathcal{G}}}
\def\gH{{\mathcal{H}}}
\def\gI{{\mathcal{I}}}
\def\gJ{{\mathcal{J}}}
\def\gK{{\mathcal{K}}}

\def\gN{{\mathcal{N}}}

\def\gR{{\mathcal{R}}}

\def\gV{{\mathcal{V}}}
\def\gW{{\mathcal{W}}}
\def\gX{{\mathcal{X}}}
\def\gY{{\mathcal{Y}}}

\def\sN{{\mathbb{N}}}

\def\sR{{\mathbb{R}}}

\def\emLambda{{\Lambda}}

\def\emDelta{{\Delta}}

\DeclareMathOperator*{\argmin}{arg\,min}

\newcommand\extrafootertext[1]{%
    \bgroup
    \renewcommand\thefootnote{\fnsymbol{footnote}}%
    \renewcommand\thempfootnote{\fnsymbol{mpfootnote}}%
    \footnotetext[0]{#1}%
    \egroup
}
\usepackage[textsize=tiny]{todonotes}
\usepackage{fullpage}

\usepackage{float}
\usepackage{titlesec}

\titleformat{\section}
  {\normalfont\large\bfseries}{\thesection}{1em}{}
  
\titleformat{\subsection}
  {\normalfont\normalsize\bfseries}{\thesubsection}{1em}{}
  
\titleformat{\subsubsection}
  {\normalfont\normalsize\bfseries}{\thesubsubsection}{1em}{}
  
\titlespacing*{\section}{0pt}{5pt plus 1pt minus 1pt}{2pt plus 1pt minus 1pt}
\titlespacing*{\subsection}{0pt}{4pt plus 1pt minus 1pt}{2pt plus 1pt minus 1pt}

\titlespacing*{\paragraph}{0pt}{2pt plus 1pt minus 1pt}{5pt plus 1pt minus 1pt}

\titlespacing*{\subsubsection}{0pt}{4pt plus 1pt minus 1pt}{1pt plus 1pt minus 1pt}

\labelwidth\leftmargini\advance\labelwidth-\labelsep \labelsep 5pt

\def\@listi{\leftmargin\leftmargini}
\def\@listii{\leftmargin\leftmarginii
   \labelwidth\leftmarginii\advance\labelwidth-\labelsep
   \topsep 2pt plus 1pt minus 0.5pt
   \parsep 1pt plus 0.5pt minus 0.5pt
   \itemsep \parsep}
\def\@listiii{\leftmargin\leftmarginiii
    \labelwidth\leftmarginiii\advance\labelwidth-\labelsep
    \topsep 1pt plus 0.5pt minus 0.5pt
    \parsep \z@ \partopsep 0.5pt plus 0pt minus 0.5pt
    \itemsep \topsep}
\def\@listiv{\leftmargin\leftmarginiv
     \labelwidth\leftmarginiv\advance\labelwidth-\labelsep}
\def\@listv{\leftmargin\leftmarginv
     \labelwidth\leftmarginv\advance\labelwidth-\labelsep}
\def\@listvi{\leftmargin\leftmarginvi
     \labelwidth\leftmarginvi\advance\labelwidth-\labelsep}

\usepackage{authblk}

\title{\textbf{\Large{
Cardinality-Constrained Bilevel Capacity Expansion
}}}

\author
{\normalsize
Lei Guo$^{1}$
\qquad  Jiayang Li$^{2, *}$
}

\date{}

\begin{document}
\extrafootertext{$^{1}$School of Business, East China University of Science and Technology. $^2$Department of Data and Systems Engineering, The University of Hong Kong. $^*$Corresponding author; email: \texttt{jiayangl@hku.hk}.
}

\maketitle

\begin{abstract}
As a fundamental problem in transportation and operations research, the bilevel capacity expansion problem (BCEP) has been extensively studied for decades. In practice, BCEPs are commonly addressed in two stages: first, preselect a small set of links for expansion; then, optimize their capacities. However, this sequential and separable approach can lead to suboptimal solutions as it neglects the critical interdependence between link selection and capacity allocation. In this paper, we propose to introduce a cardinality constraint into the BCEP to limit the number of expansion locations rather than fixing such locations beforehand. This allows us to search over all possible link combinations within the prescribed limit, thereby enabling the \emph{joint} optimization of both expansion locations and capacity levels. The resulting cardinality-constrained BCEP (CCBCEP) is computationally challenging due to the combination of a nonconvex equilibrium constraint and a nonconvex and discontinuous cardinality constraint. To address this challenge, we develop a penalized difference-of-convex (DC) approach that transforms the original problem into a sequence of tractable subproblems by exploiting its inherent DC structure and the special properties of the cardinality constraint. We prove that the method converges to approximate Karush–Kuhn–Tucker (KKT) solutions with arbitrarily prescribed accuracy. Numerical experiments further show that the proposed approach consistently outperforms alternative methods for identifying practically feasible expansion plans investing only a few links, both in solution quality and computational efficiency.

\vspace{3pt}
    \vspace{1em}
\end{abstract}

\section{Introduction}\label{sec:Intro}

A bilevel capacity expansion problem (BCEP) considers how a traffic planner can most effectively expand a congested road network, while accounting for the self-interested route choices of travelers \citep{leblanc1975algorithm, abdulaal1979continuous}. The emergence of BCEPs in the 1970s was closely linked to the rapid urbanization and economic growth experienced in many countries during the postwar period. As cities expanded and populations grew, travel demand quickly outpaced the capacity of existing transportation infrastructure. BCEP has since become a useful approach for assisting traffic planners in infrastructure investment decisions, thereby effectively mitigating congestion and enhancing the overall efficiency of transportation networks \citep[see, e.g.,][]{dantzig1979formulating,marcotte1986network,friesz1992simulated,yang2001transport,meng2001equivalent,ban2006general,josefsson-Patriksson2007,li2012,gairing2017complexity, li2022differentiable,guo2025penalized}.

The focus of the present study is a challenging BCEP scenario in which a large number of links are practically eligible for expansion --- a situation that often arises in real-world networks. For instance, in the Anaheim network \citep{TransportationNetworks}, a medium-sized network widely used in transportation research, hundreds of links along one interstate (I-5) and four state highways (SR-22, SR-55, SR-57, SR-91) have been identified as candidates for expansion in recent years \citep{octa_i5_2025, caltrans_sr55_2022, dya_sr22_2025, usdotsr91_2025, ctc_0p670_2018}. In larger, regionally coordinated planning efforts, the number of links eligible for expansion can be even higher. In such scenarios, a key challenge is that the optimal solution to the corresponding BCEP may recommend expanding \textit{every} eligible link, which may be theoretically optimal but rarely practical. Specifically, expanding all selected links simultaneously would severely disrupt network operations and create acute congestion during construction, while expanding them sequentially would greatly extend the construction timeline --- potentially over many years or even decades --- leading to prolonged inconvenience and inefficiency.

Due to the impracticality of expanding all eligible links, BCEP studies on real-world networks typically assume that a \textit{small} set of candidate links has been selected externally beforehand for potential capacity expansion \citep[see, e.g.,][]{suwansirikul1987equilibrium, friesz1992simulated, chow2010faster, guo2025penalized}. However, few have systematically investigated how such a subset can be determined, except for some heuristics such as performing an interpretable prescreening of all eligible links. 
For instance, planners may begin by evaluating the marginal effect of expanding unit capacity on each candidate link --- specifically, the reduction in total congestion that such an investment would yield --- which can be obtained through sensitivity analysis at the current equilibrium state \citep{dafermos1983iterative, tobin1988sensitivity, yang2007sensitivity, li2020end}. Links whose expansion would paradoxically increase total congestion at equilibrium --- thus triggering the \citet{braess1968paradoxon} paradox --- can then be excluded from consideration \citep{roughgarden2005selfish, park2011detecting, bagloee2014heuristic}. Conversely, links where a small investment yields substantial congestion reduction may be included in the expansion plan.  However, the prescreening approach has several notable disadvantages. First, it relies solely on local sensitivity information, which may not reflect the true impact of larger, coordinated capacity expansions. Second, by assessing the effect of investing in each link independently, the approach does not account for network-wide changes that might emerge once multiple links are expanded together.

\subsection{Our approach}

Unlike conventional approaches that first prescreen candidate links and then solve a standard BCEP on the selected subset, we propose to optimize link selection and capacity expansion simultaneously within a unified framework. To this end, we introduce a cardinality constraint into the standard BCEP formulation, which explicitly limits the maximum number of links that can be expanded. The optimal solution to the resulting model, referred to as a cardinality constrained BCEP (CCBCEP), can thus be interpreted as the most effective selection of up to a specified number of links to maximize congestion relief while accounting for monetary cost.

Despite the advantages of CCBCEP in jointly identifying optimal expansion locations and levels, the problem is challenging to solve. Even the classic BCEP is difficult: since travelers' route choices in response to capacity expansion are typically predicted by the user equilibrium (UE) condition \citep{wardrop1952road}, the problem is formulated as a bilevel program and is known to be NP-hard due to the non-convexity and non-regularity of the equilibrium constraint \citep{colson2007overview, gairing2017complexity}. CCBCEP is even more complex, as it involves both the equilibrium constraint and the cardinality constraint. The main challenge of the cardinality constraint lies in the discontinuity and non-convexity of the cardinality function at points where variables transition from non-zero values to zero.

To devise an algorithm for solving such problems, we first reformulate the lower-level UE problem using a special gap function, which is non-positive if and only if the UE condition is satisfied. Under mild conditions, this gap function is continuously differentiable and possesses a difference-of-convex (DC) structure, that is, it can be expressed as the difference of two convex functions. By penalizing this gap function into the objective, the equilibrium constraint is relaxed to a much simpler flow-conservation constraint, which requires that the total flow on all routes for each OD pair matches the demand. Despite this relaxation, the penalized problem remains challenging. The turning point is that by linearizing one convex component of the objective, the resulting approximate problem becomes tractable. 

A key observation is that the objective in these approximate problems is separable with respect to the lower-level flow variables when the upper-level expansion variables are fixed, and vice versa. We find and prove that such a separability (i) reduces the minimization over the flow-conservation constraint to a standard traffic assignment problem, which can be efficiently solved using existing UE algorithms, and (ii) reduces the minimization over the cardinality set to a set of one-dimensional strongly convex minimization problems, which can be efficiently solved using Newton's methods or bisection methods. Motivated by these properties, we propose an algorithm for the approximate problem that alternates between optimization over the flow-conservation constraint and the cardinality constraint. Based on the algorithm for solving the approximate problem, we eventually develop an algorithm for the original CCBCEP by sequentially solving these approximate problems and adaptively updating the penalty parameter for the equilibrium constraint.

In addition to developing the model and algorithm, our contribution also includes a theoretical justification of the algorithm's convergence to approximate KKT stationary solutions under mild conditions, as well as a thorough validation of its performance through extensive numerical experiments. Our numerical results demonstrate that, across all tested instances, the solution produced by our model (i) can yield expansion plans that only invests a very small subset of links but are still effective, and (ii) consistently outperforms alternative approaches in solution quality and computational efficiency.

\subsection{Related work}
\label{sec:related}

In the literature, numerous algorithms have been developed for classic BCEPs in which the candidate expansion links are already specified. These algorithms can generally be classified into four categories: (i) indirect algorithms that solve an approximation model; (ii) equilibrium solution-based algorithms; (iii) equilibrium condition-based algorithms; (iv) value function-based algorithms.

The first class seeks to solve a simplified problem whose solution is expected to approximate that of the original BCEP. Two representative examples are the iterative optimization–assignment (IOA) algorithm \citep{tan1979hybrid}, which alternates between optimizing the designer's decisions and the travelers' responses while fixing the other party's decisions \citep{universite1981design, friesz1985properties, marcotte1986network, marcotte1992efficient}, and the system-optimal approach \citep{dantzig1979formulating}, which simplifies the BCEP by assuming that travelers cooperate with the designer to achieve the same objective. Such schemes may yield good approximate solutions in certain scenarios.

For equilibrium solution-based algorithms, the equilibrium solution at the lower level is treated as an implicit best response mapping with respect to upper decision variables, and embed into the upper-level objective, reducing the bilevel problem to a single-level implicit composition function minimization. Upper decision variables are updated by computing and comparing values of the composite function, such as the Hooke–Jeeves algorithm \citep{abdulaal1979continuous} and simulated annealing \citep{friesz1992simulated}, or by computing the gradient which depends on the Jacobian of the best response function \citep{friesz1990, yang2005mathematical, josefsson-Patriksson2007}. A conventional approach to computing this Jacobian is to implicitly differentiate the UE conditions, which is often computationally prohibitive due to the need to invert large matrices \citep{tobin1988sensitivity, dafermos1988sensitivity, yang2007sensitivity}.  More recently, automatic differentiation (AD) \citep{griewank1989automatic} has emerged as a promising tool for scaling descent algorithms for BCEPs. Specifically, it has been shown that one can compute the gradient of a bilevel program by applying AD to a smooth algorithm that solves the lower-level UE problem \citep{li2020end, li2022differentiable}. However, a significant challenge remains in the substantial memory required to store the entire computational process. We refer the readers to \citet{mckenzie2024three} and \citet{yu2025scalable} for recent advances aimed at overcoming this limitation.

Instead of directly addressing the implicit equilibrium solution, equilibrium condition-based algorithms first formulate the lower-level equilibrium requirement as a set of nonlinear complementarity constraints or, equivalently, a variational inequality problem \cite{dafermos1980traffic}, and then solve these reformulations. Due to a huge number of constraints in these reformulations, \cite{marcotte1983} proposed a constraint accumulation method, whereas \cite{lim2002transportation} proposed a path generation-based method. As other examples, \cite{wang2010} proposed to linearize all nonlinear functions in the reformulation and solve the resulting mixed linear programming problem, and \cite{li2012} utilized the gap function of a variational inequality and proposed a global optimization method.

The final class of algorithms recasts the lower‐level UE conditions through a gap function whose value is nonnegative and zero only when the UE conditions hold. The approach avoids introducing huge number of constraints compared with the equilibrium condition-based reformulation. By showing the gap function is continuously differentiable, standard augmented Lagrangian methods are applied to solve this reformulation \citep{meng2001equivalent,yang2005mathematical}. More recently, \citet{guo2025penalized} showed that this gap function is a DC function and then designed a penalized sequential programming scheme to achieve far better scalability.

Although a rich body of algorithms exists for BCEPs once the set of expandable links is fixed, the literature offers little guidance on how that set should be chosen in the first place. In most empirical studies on benchmark networks --- such as Sioux Falls, Anaheim, and Chicago Sketch --- the authors choose a small subset of links for expansion without much explanation of the selection process \citep{suwansirikul1987equilibrium, friesz1992simulated, chow2010faster, guo2025penalized}.
Recently, \citet{guo2024cardinality} investigates a similar model where links for congestion pricing and the toll on these links are optimized together. However, the algorithm by \citet{guo2024cardinality} cannot be adopted to solve CCBCEPs, as it crucially relies on the multi-convexity structure --- a property that these problems lack. To the best of our knowledge, the present work is the first to treat link selection as an integral part of finding optimal expansion plans.

\subsection{Organization}
The remainder of this paper is organized as follows. Section \ref{sec:Model} introduces the novel CCBCEP model. In Section \ref{sec:newmodel}, we reformulate the model and analyze its theoretical properties. Section \ref{sec:Algo} presents the proposed algorithm for solving the CCBCEP and establishes its convergence under standard conditions. Extensive numerical experiments on publicly available networks are reported in Section \ref{sec:num}. Finally, Section \ref{sec:Conclusion} concludes the paper with key findings and future research directions.

\subsection{Notation}

We use $\sR$, $\sR_+$, and $\sN$ to denote the set of real, non-negative real, and natural numbers. The indicator function of a condition $\gE$ is denoted by $\mathbf{1}_{\gE}$, which equals $1$ if $\gE$ is true and $0$ otherwise. For a vector $\va \in \sR^n$, we denote its support as $\supp{(\va)} = \{i: \eva_i > 0\}$ and its $\ell_1$, $\ell_2$, and $\ell_{\infty}$ norms as $\|a\|_1$, $\|a\|_2$, and $\|a\|_{\infty}$, respectively. 
For two vectors $\va, \vb \in \sR^n$, their inner product is denoted as $\langle \va, \vb \rangle$. For a finite set $\gA$, we write $|\gA|$ as its number of elements and $2^{\gA}$ as the set of its subsets. Given a closed set $\gX \in \sR^n$, we define its diameter as $\diam(\gX) = \max_{\vx^1, \vx^2 \in \gX} \|\vx^1 - \vx^2\|_2$.
For any $\va \in \sR^n$ and a nonempty set $\gX \subseteq \sR^n$, we use $\operatorname{dist}(\va, \gX) = \inf\{ \|\va - \vb\|_2 : \vb \in \gX\}$ to denote the distance from $\va$ to $\gX$. Given any closed set $\gX \subseteq \sR^n$ and $\bar \vx \in \gX$, the limiting normal cone to $\gX$ at $\bar{\vx}$ is defined as
\[
\cN_{\gX}(\bar{\vx}) = \left\{\bar{\vz}: \exists~ \vx^k\to \bar{\vx}~\text{and}~\vz^k\to \bar{\vz}~\text{with}~\vx^k \in \gX, \vz^k \in \sR^n~\text{and}~\lim_{\vx\to \vx^k}\frac{\langle \vz^k, \vx - \vx^k\rangle}{\|\vx-\vx^k\|_2} \leq 0, \ \forall k\right\}
\]
When $\gX$ is a convex and closed set, the limiting normal cone $\cN_{\gX}(\bar{\vx})$ reduces to $\{\bar \vz \in \sR^n: \langle \bar \vz, \vx - \bar \vx\rangle \leq 0,\ \forall \vx\in \gX\}$, which is the standard normal cone from convex analysis.

\section{Problem Setting}\label{sec:Model}

Consider a transportation network modeled as a strongly connected directed graph $\gG(\gN, \gA)$, where $\gN$ and $\gA$ are the sets of nodes and links, respectively. The set of all origin-destination (OD) pairs is denoted by $\gW \subseteq \gN \times \gN$, and the set of paths between all OD pairs is denoted as $\gR \subseteq 2^{\gA}$. For each OD pair $w \in \gW$, let $\gR_w \subseteq \gR$ be the set of routes connecting $w$, and for each route $r \in \gR$, let $\gA_r \subseteq \gA$ be the set of links composing $r$. Further, let $\emLambda_{w,r}$ as the OD-route incidence with $\emLambda_{w,r} = 1$ if the route $r \in \gR_w$ and 0 otherwise, and $\emDelta_{a,r}$ as the link-route incidence, with $\emDelta_{a,r} = 1$ if $a \in \gA_r$ and 0 otherwise. Let $\mDelta = (\emDelta_{a,r})_{a \in \gA, r \in \gR}$ and $\mLambda = (\emLambda_{w,r})_{w \in \gW, r \in \gR}$. Suppose travel demand is given by $\vd = (\evd_w)_{w \in \gW}$, where $\evd_w$ denotes the number of travelers for OD pair $w$. Let $\vh = (\evh_r)_{r \in \gR}$ and $\vv = (\evv_a)_{a \in \gA}$ denote the path and link flow patterns, respectively. Their feasible regions are then $\gH = \{\vh \in \sR^{|\gR|}: \mLambda \vh = \vd, \ \vh \geq 0\}$ and $\gV = \{\vv \in \sR^{|\gA|}: \vv = \mDelta \vh, \ \vh \in \gH \}$, respectively.

A capacity expansion problem concerns a network designer seeking to expand the capacity of existing links. Let $\vy = (y_a)_{a \in \gA}$ denote the planner's decision, where each $y_a$ is the additional capacity assigned to link $a$. Suppose that the decision vector must lie within the feasible region $\gY = \{\vy \in \sR^{|\gA|}: 0 \leq \evy_a \leq u_a, \ \forall a\in \gA\}$, where $u_a \in \gR$ represents the maximum permissible capacity increase on link $a$, for instance, due to geographical or engineering constraints. If $u_a = 0$, it means the link cannot be expanded. Let $t(\vy,\vv)$ denote the link travel time function --- which is determined by the link flow $\vv$ and affected by the added capacity $\vy$ --- and $G(\vy)$ the expansion cost function.  Throughout this study, we assume that given any $\vy \in \gY$, travelers' route choices are predicted by user equilibrium (UE), i.e., no traveler can unilaterally switch to a route with lower cost \citep{wardrop1952road}. By \citet{dafermos1980traffic}, a link flow pattern $\vv \in \gV$ satisfies the UE condition if and only if the following variational inequality (VI) holds:
\begin{equation}\label{vip}
\langle t(\vy,\vv),~ \vv' - \vv \rangle \geq 0, \quad \forall~\vv' \in \gV.
\end{equation}
For notational simplicity, we denote the solution set to this VI by $\gV^*(\vy)$.

Under the above setting, it is often assumed that the designer aims to minimize
\begin{eqnarray*}
\begin{array}{rl}
F(\vy,\vv) := \langle t(\vy,\vv), \vv \rangle + \eta \cdot G(\vy)
\end{array}
\end{eqnarray*}
which is a weighted sum of the total travel time experienced by travelers at UE and the expansion cost, with an appropriate $\eta > 0$ as the weighting coefficient \citep{abdulaal1979continuous}.
If the feasible region for $\vy$ is set as $\gY$, however, there would be a practical challenge: the resulting solution may suggest adding capacity to every link $a$ that can be expanded (i.e., those satisfying $u_a > 0$), leading to an expansion plan far from feasible in practice. A commonly adopted strategy to address this issue is to pre-select a \textit{small} subset of links, $\hat{\gA} \subseteq \gA$, for potential expansion. This leads to the following BCEP formulation frequently used in practice \citep{suwansirikul1987equilibrium, friesz1992simulated, chow2010faster, guo2025penalized}:
\begin{eqnarray}\label{cndp-A}
\begin{array}{rl}
\displaystyle{\rm BCEP(\hat{\gA})}~~~~~~\min_{\vy, \vv} & ~F(\vy,\vv),\\ [5pt]
{\rm s.t.} & ~\vy\in \gY, \quad \vv \in \gV^*(\vy), \\ 
&~\evy_a=0, \quad \forall  a \in \gA \setminus \hat{\gA}.\\ 
\end{array}
\end{eqnarray}
However, there are few widely recognized criteria for selecting such links beyond heuristic rules. A relatively reasonable heuristic is to rank links by the marginal decrease in the objective per unit of expansion on each link $a$. This requires additional assumptions to ensure that the UE link flow pattern is unique and that the equilibrium mapping $v^*: \vy \mapsto \gV^*(\vy)$ is continuously differentiable \citep[see, e.g.,][for such regularities]{patriksson2004sensitivity, yang2005mathematical}. Further defining $F^*(\vy) = F(\vy, v^*(\vy))$, the marginal effect at the status quo can then be computed as
\begin{align}
    e_a = -\nabla_{\evy_a} F^*(\vy) \mid_{\vy = \vzero}.
    \label{eq:sensitivity}
\end{align}
Note that a negative $e_a$ implies that investing in link $a$ would actually increase total travel time, reflecting the Braess paradox \citep{braess1968paradoxon}. Such links should generally be excluded from the expansion plan.  On the contrary, a large positive $e_a$ suggests that a small expansion of link $a$ could lead to a relatively large reduction in the objective function. A network designer might therefore select a small subset of links with high $e_a$ for expansion. However, this approach may lead to suboptimal decisions, especially in large and complex networks, as it relies solely on local sensitivity information and ignores interactions among links. This challenge motivates us to develop the proposed model, which shall be introduced in the next section.

\section{Our Model} \label{sec:newmodel}

In this paper, we propose to simultaneously determine the links to expand and their corresponding expansion levels in planning capacity enhancement. Specifically, instead of pre-selecting a subset of links $\gA$ with a small cardinality $\tau = |\supp(\gA)|$, we hope to search over all possible expansion plans that add capacity to at most $\tau$ links, namely
\[
\Upsilon_\tau=\{\vy\in \gY: |\supp(\vy)|\leq \tau\}.
\]
By integrating this cardinality constraint set with the classic formulation, we propose the following cardinality-constrained BCEP (CCBCEP):
\begin{eqnarray}\label{p}
\begin{array}{rl}
\displaystyle{\text{CCBCEP}}~~~~~~\min_{\vy,\vv} & ~F(\vy,\vv), \\ 
{\rm s.t.} & ~\vy\in \Upsilon_\tau,\quad \vv \in \gV^*(\vy).
\end{array}
\end{eqnarray}
In contrast to $\text{BCEP}(\hat{\mathcal{A}})$, our new model relaxes the assumption of pre-determined expansion locations and instead allows the locations to be selected optimally. Plainly, its optimal solution can be interpreted as the most effective selection of up to $\tau$ links to maximize congestion relief while considering the monetary cost.

Despite its practical appeal, CCBCEP is difficult to solve because it involves two challenging constraints: the equilibrium condition $\vv \in \gV^{*}(\vy)$ and the cardinality restriction $\vy \in \Upsilon_{\tau}$. The equilibrium constraint alone already makes the problem strongly NP-hard \citep{gairing2017complexity}. The cardinality constraint is also formidable, even when one attempts to handle it through penalization, because it would introduce the term
$$|\supp(\vy)|= \sum_{a \in \gA} \mathbf{1}_{\evy_a > 0},$$
into the objective function, which jumps discontinuously whenever a component $\evy_a$ crosses zero.

In the remainder of the section, we will first establish a reformulation of CCBCEP \eqref{p} (Section~\ref{sec:reformulation}) and then analyze some key properties of the reformulated problem (Section~\ref{sec:property}).

\subsection{Equivalent reformulation}\label{sec:reformulation}

As a first step toward tackling CCBCEP \eqref{p}, we reformulate the problem under a set of standard assumptions that are commonly used in BCEP studies \citep[e.g.,][]{leblanc1975algorithm,abdulaal1979continuous,marcotte1986network,meng2001equivalent,wang2013} 

\begin{assumption}
    \label{ass:1}
    The link cost function $t(\vv;\vy) = (t_a(\evv_a;\evy_a))_{a \in \gA}$, and each $t_a(\evy_a,\evv_a)$ is continuously differentiable in $(\vv, \vy)$, and strictly increasing in $\evv_a$ for all $\evy_a \geq 0$.
\end{assumption}

\begin{assumption}
    \label{ass:2}
    The expansion cost function $G(\vy) = \sum_{a \in \gA} G_a(\evy_a)$, and each $G_a(\evy_a)$ is nonnegative, continuously differentiable and strictly increasing in $\evy_a$.
\end{assumption}

Among others, Assumption \ref{ass:1} ensures that for any $\vy \in \gY$, the VI problem \eqref{vip} admits a unique solution \citep{dafermos1980traffic}, and can be reformulated via \citet{beckmann1956studies}'s transformation, which suggests that any $\vv \in \gV$ is the solution if and only if it solves the following convex program
\begin{equation}\label{ue}
\begin{array}{rl}
\displaystyle \vv = \argmin_{\vv' \in \gV} & \displaystyle~f(\vv'; \vy) := \sum_{a\in {\gA}}\int_0^{\evv_a'} t_a(\evy_a,w) \dif w.
\end{array}
\end{equation}
Denote the value and gap function for this convex program as
\begin{equation*}
    g(\vy) := \inf\{f(\vy,\vv): \vv\in \gV\} \quad \text{and} \quad \varphi(\vy,\vv) := f(\vv;\vy)-g(\vy),
\end{equation*}
respectively.
It is then easy to see for all $\vy \in \gY$ and $\vv \in \gV$, the gap function $\varphi(\vy,\vv)$ is always nonnegative, and would be nonpositive if and only if $\vv \in \gV^*(\vy)$. Accordingly, CCBCEP \eqref{p} can be transformed into the following \textit{equivalent} problem:
\begin{eqnarray}\label{equi-p}
\begin{array}{rl}
\displaystyle{\text{CCBCEP-E}}~~~~~~\displaystyle \min_{\vy, \vv} & ~~F(\vy,\vv), \\ 
{\rm s.t.} & ~~\vy\in \Upsilon_\tau,\quad \vv\in \gV,\\
           & ~~\varphi(\vy,\vv) \leq 0.
\end{array}
\end{eqnarray}
Although the reformulation admits a single-level structure, it still poses challenges on three fronts. First, the cardinality constraint set $\vy\in\bm{\Upsilon}_\tau$ is a union of polyhedral sets (see Proposition \ref{prop-cardiset}) and thus introduces combinatorial complexity. 
Second, evaluating the constraint $\varphi(\vv;\vy)\leq 0$  requires computing the value function $g(\vy)$, which involves solving a UE problem. With such a constraint, it is not easy to develop scalable algorithms. Moreover, at any feasible point of CCBCEP-E~\eqref{equi-p}, the gap function $\varphi(\vv;\vy)$ must be exactly zero, since by construction $\varphi(\vv;\vy)\ge 0$ for all $(\vy,\vv)\in \gY\times \gV$ and feasibility imposes $\varphi(\vv;\vy)\le 0$. Hence, the strict inequality $\varphi(\vv;\vy)<0$ never holds at a feasible point. As a result, the Mangasarian–Fromovitz Constraint Qualification (MFCQ) fails everywhere for CCBCEP-E~\eqref{equi-p}, because there is no feasible direction that further decreases $\varphi$ at the boundary. Accordingly, standard nonlinear programming algorithms that rely on MFCQ (or similar qualifications) are not directly applicable.

\subsection{Properties}\label{sec:property}

We now proceed to investigate several key properties of CCBCEP-E~\eqref{equi-p} that will facilitate both algorithm design and convergence analysis.
First, observe that without the cardinality constraint, CCBCEP-E~\eqref{equi-p} reduces to the value function-based formulation of the classic BCEP, a case recently studied by \citet{guo2025penalized}, who established the following result (see their Proposition 2).
\begin{proposition}\label{prop:gradient}
Under Assumptions \ref{ass:1} and \ref{ass:2}, it holds that (i) both $F(\vy, \vv)$ and $f(\vy,\vv)$ are continuously differentiable; (ii) $g(\vy)$ is also continuously differentiable, and its gradient reads
    \begin{equation}\label{gra-formula}
    \nabla g(\vy) = \left(\int_0^{v_a} \nabla_{\evy_a} t_a(\evy_a,w) \dif w\right)_{a\in \gA},
    \end{equation}
    where $(\evv_a)_{a\in \gA}$ is the unique solution to Problem \eqref{ue} (i.e., the unique link flow pattern at UE).
\end{proposition}

Proposition \ref{prop:gradient} indicates that to compute $\nabla g(\vy)$, it suffices to first obtain the unique $\vv \in \gV^*(\vy)$ using any standard UE algorithm, and then evaluate the gradient according to Equation~\eqref{gra-formula}.

We next check the geometry of the cardinality constraint set $\Upsilon_\tau$, which is clearly non-convex unless $\tau = |\gA|$: for any two points in $\Upsilon_\tau$ with distinct supports, their convex combination generally has more than $\tau$ positive components. Yet, as established in the following proposition, $\Upsilon_\tau$ can be represented as a union of finitely many polyhedral convex sets.

\begin{proposition}\label{prop-cardiset}
The cardinality constraint set $\Upsilon_\tau$ can be expressed as the union of finitely many polyhedral convex sets, specifically,
\begin{equation}\label{prop:union}
\Upsilon_\tau = \bigcup_{\gI\in \Xi_\tau} \left\{\vy\in \gY: \evy_a=0, \ \ a\in \gA\setminus \gI\right\}, \quad \text{where}~\Xi_\tau = \{\gI \subset \gA : |\gI| = \tau\}.
\end{equation}
\end{proposition}
\begin{proof}
See Appendix~\ref{app:prop-cardiset} for the proof.
\end{proof}

Proposition~\ref{prop-cardiset} ensures that $\Upsilon_\tau$ is a closed set, so any accumulation point of a sequence within $\Upsilon_\tau$ also belongs to $\Upsilon_\tau$, a property that is crucial for the convergence analysis of algorithms.

Taken together, the two propositions imply that CCBCEP-E~\eqref{equi-p} is a continuously differentiable program defined over a closed set. To further leverage the convexity and concavity properties of the problem, which can often substantially improve algorithmic efficiency \citep{boyd2004convex}, we impose the following additional assumption.

\begin{assumption}\label{assu-convex} For each $a\in \gA$, (i)$\int_0^{\evv_a} t_a(\evy_a,w) \dif w$ is convex in $(\evy_a,\evv_a)$; (ii)  $t_a(\evy_a,\evv_a) \cdot 
\evv_a$ is convex in $(\evy_a, \evv_a)$; (iii) $G_a(\evy_a)$ is convex in $\evy_a$.
\end{assumption}

It can be checked that Assumption \ref{assu-convex}-(i) and -(ii) holds as long as the travel time cost function follows the Bureau of Public Roads (BPR) form, i.e.
\begin{equation}\label{bpr}
t_a(\evy_a,\evv_a) = t_{a,0}\cdot \left(1 + 0.15\cdot\left(\frac{\evv_a}{\evc_a + \evy_a}\right)^4\right),\quad \forall a\in \gA,
\end{equation}
where $t_{a,0}$ is the free-flow travel time and $c_a$ is the original capacity of link $a \in \gA$. Meanwhile, Assumption~\ref{assu-convex}-(iii), which requires the convexity of $G_a(y_a)$, simply means that the marginal cost of capacity expansion is non-decreasing --- i.e., each additional unit of capacity becomes progressively more expensive. This reflects typical real-world scenarios. Under these assumptions, the following properties hold \citep[Proposition 3]{guo2025penalized}.

\begin{proposition}\label{prop:convex}
Under Assumption~\ref{assu-convex},
(i) $F(\vy, \vv)$ is convex with respect to $(\vy, \vv)$;
(ii) $\varphi(\vy,\vv) = f(\vy,\vv) - g(\vy)$ admits a difference-of-convex (DC) structure in the sense that $f(\vy,\vv)$ is convex in $(\vv, \vy)$ and $g(\vy)$ is convex in $\vy$.
\end{proposition}

The DC structure of $\varphi(\vy,\vv)$ is a key to our algorithm to be developed in the next section.

\section{Solution Algorithm}\label{sec:Algo}

We are now ready to develop an algorithm for solving CCBCEP-E~\eqref{equi-p}. While optimal solutions are desirable in theory, they are computationally intractable for nonlinear and non-convex problems of this form on large-scale networks \citep{nocedal1999numerical}. In this section, we present an algorithm capable of computing $\epsilon$-approximate KKT stationary solutions (shall be formally defined later) for arbitrarily small $\epsilon$. We begin by outlining the general idea behind the algorithm (Section~\ref{sec:idea}), followed by implementation details and a convergence analysis (Section~\ref{sec:implementation}).

\subsection{Idea}
\label{sec:idea}

As the first step,  we propose the following partially penalized approximation (PPA) of CCBCEP-E~\eqref{equi-p} to overcome the computational difficulties arising from the failure of MFCQ in CCBCEP-E~\eqref{equi-p}, in which the constraint $\varphi(\vy,\vv) \leq 0$ is incorporated into the objective function as a penalty term:
\begin{eqnarray}\label{penalty_p}
\begin{array}{rl}
\text{PPA}_\rho~~~~~~\displaystyle \min_{\vy,\vv} & ~~F(\vy,\vv) +\rho \cdot \varphi(\vy,\vv), \\ 
{\rm s.t.} & ~~\vy\in \Upsilon_\tau,\quad \vv\in \gV,
\end{array}
\end{eqnarray}
where $\rho > 0$ is a penalty parameter. An advantage of PPA$_{\rho}$~\eqref{penalty_p} is that its feasible region consists of two separable constraint sets, $\Upsilon_\tau$ and $\gV$, which potentially enables the development of efficient decomposition algorithms.  Yet, it should be noted that the solution to PPA$_{\rho}$~\eqref{penalty_p} may not satisfy the equilibrium constraint $\vv \in \gV^*(\vy)$ exactly, as it is a penalty-based relaxation. Later, we will discuss how to appropriately choose and adjust the penalty parameter $\rho$ so that the solution to PPA$_{\rho}$~\eqref{penalty_p} gradually becomes feasible and optimal for the original CCBCEP.

We then present the following key result, which shows that minimizing a link-separable function in $\vy$ over the cardinality-constrained set $\Upsilon_\tau$ reduces to solving a set of tractable one-dimensional optimization problems.

\begin{proposition}\label{prop:closedsolution}
Given a class of functions $\{\phi_a: \sR \to \sR\}_{a \in \gA}$, consider the following optimization problem
\begin{equation} \label{prop-opt}
\min_{\vy} \sum_{a\in \gA} \phi_a(y_a),\quad  {\rm s.t.}\ \  \vy\in \Upsilon_\tau.
\end{equation}
For each $a \in \gA$, suppose that $y_a^*$ is an optimal solution of the following one-dimensional optimization problem
\[
\min\ \phi_a(y_a), \quad {\rm s.t.}\ 0\leq y_a \leq u_a.
\]
Let $\bar{\gA}$ be the index set of the $\tau$ smallest values of $\{\phi_a(y_a^*)-\phi_a(0):a\in \gA\}$.
Then Problem \eqref{prop-opt} has the following closed-form solution $\bar{\vy}^* = (\bar{\evy}_a^*)_{a\in\gA}$, where
\[
\bar{\evy}_a^* = 
\begin{cases}
    y_a^*, \quad &\text{if}~a\in \bar{\gA}, \\
    0, \quad &\text{otherwise}.  
\end{cases}
\]
\end{proposition}

\begin{proof}
    See Appendix \ref{app:closedsolution} for the proof.
\end{proof}

Importantly, Proposition~\ref{prop:closedsolution} motivates us to reformulate PPA$_{\rho}$~\eqref{penalty_p} so that its objective becomes link-separable in $\vy$. Recall that the objective consists of two parts: $F(\vy, \vv)$, which is convex and link-separable in $\vy$, and $\varphi(\vy,\vv)$, which admits a DC structure by Proposition~\ref{prop:convex}. To further facilitate algorithmic handling, we propose decomposing $\varphi(\vy,\vv)$ as the difference of two strongly convex functions by introducing a regularization term $\beta \cdot \|(\vy, \vv)\|_2^2$, which leads to
\begin{equation}\label{decom}
\varphi(\vy,\vv) = \big[f(\vy,\vv) + \beta \cdot \|(\vy, \vv)\|_2^2 \big] - \big[ g(\vy) + \beta \cdot \|(\vy, \vv)\|_2^2 \big],
\end{equation}
where $\beta > 0$ is a regularization parameter --- how to choose and adjust $\beta$ would be discussed later.
Among all components of the objective in PPA$_{\rho}$~\eqref{penalty_p}, one may then notice that only $g(\vy) + \beta \cdot \|(\vy, \vv)\|_2^2$ is \textit{not} link-separable in $y$. However, this can be addressed by linearization. Specifically, supposing that the current solution is $(\vy^k, \vv^k)$, linearizing $g(\vy) + \beta \cdot \|(\vy, \vv)\|_2^2$ in Equation~\eqref{decom}, which is convex with respect to $(\vy, \vv)$, yields a strongly convex upper approximation for the gap function $\varphi(\vy,\vv)$:
\begin{equation}\label{convex-appro}
\begin{split}
\varphi(\vy,\vv) &\leq \Phi(\vy, \vv; \vy^k,\vv^k) + \beta \cdot \|(\vy-\vy^k, \vv-\vv^k)\|_2^2,
\end{split}
\end{equation}
where, for notational simplicity,
$$
\Phi(\vy, \vv; \vy^k,\vv^k) := f(\vy,\vv) -[g(\vy^k) + \nabla g(\vy^k)^{\rm T} \cdot (\vy-\vy^k)].
$$
Note that as Equation~\eqref{convex-appro} holds for all $\beta \geq 0$, we also have
\begin{equation}
    \varphi(\vy,\vv) \leq \Phi(\vy, \vv; \vy^k, \vv^k).
    \label{eq:phi-Phi}
\end{equation}

Suppose that the penalization and the regularization parameters are now specified as  $\rho^k$ and $\beta^k$, respectively.
Based on Equation~\eqref{convex-appro}, we propose the following linearized approximation for PPA$_{\rho}$~\eqref{penalty_p}, i.e., a partially penalized and linearized approximation (PPLA) for the original CCBCEP-E~\eqref{equi-p}, at the current solution $(\vy^k, \vv^k)$:
\begin{eqnarray}\label{approP}
\text{PPLA}_k ~~~~~~ 
\begin{array}{rl}
\displaystyle \min_{\vy, \vv} & ~~\Psi^k(\vy, \vv) := F(\vy, \vv) + \rho^k \cdot \Phi(\vy, \vv; \vy^k,\vv^k) + \rho^k \cdot \beta^k \cdot \|(\vy-\vy^k, \vv-\vv^k)\|_2^2, \\ 
{\rm s.t.} & ~~\vy\in \Upsilon_\tau,\quad \vv\in \gV.
\end{array}
\end{eqnarray}
Although the non-convex constraint $\vy \in \Upsilon_\tau$ remains in PPLA$_{k}$~\eqref{approP}, a key observation is that its objective is not only strongly convex but also is link-separable in $\vy$, which implies that $\Psi^k(\vy, \vv)$ is well-suited to an alternating (block coordinate) minimization approach. Specifically, one can alternately optimize over $\vv \in \gV$ with $\vy$ fixed --- which corresponds to solving a standard traffic assignment problem with link-specific parameters \citep{beckmann1956studies} --- and over $\vy \in \gY$ with $\vv$ fixed. The latter, thanks to Proposition~\ref{prop:closedsolution}, reduces to a set of tractable one-dimensional problems.

Thus far, we have outlined the general idea behind our algorithmic approach. In the next section, we provide detailed steps for solving PPLA$_{k}$~\eqref{approP} --- the partially penalized and linearized approximation of CCBCEP-E~\eqref{equi-p} --- and explore how its solutions can be leveraged to solve CCBCEP-E~\eqref{equi-p} itself, which is an equivalent formulation of the original CCBCEP~\eqref{p}.

\subsection{Implementation and convergence analysis}
\label{sec:implementation}

\subsubsection{Algorithm for \texorpdfstring{PPLA$_{k}$}{PPLA}}
\label{sec:approP}

Based on the above discussion, we propose Algorithm~\ref{alg:ama} for solving PPLA$_{k}$~\eqref{approP}. Starting from an appropriate initial point $\vy^{k,0} \in \Upsilon_\tau$ ($\vy^k$ itself is a suitable choice), the algorithm alternately optimizes over $\vv$ and $\vy$, keeping the other variable fixed. Since PPLA$_{k}$~\eqref{approP} is non-convex, we expect that the algorithm will, in general, return a partially optimal solution --- a stronger condition than stationarity --- which we use as the termination criterion. The formal definitions of these two notions are provided below.
\begin{definition}\label{defi1}
(i) A point $(\bar{\vy},\bar{\vv})\in \Upsilon_\tau\times \gV$ is said to be partially optimal for Problem $\text{PPLA}_k$ if 
\begin{equation*}
    \begin{cases}
    \Psi^k(\bar{\vy},\bar{\vv}) \leq \Psi^k(\bar{\vy},\vv),\quad \forall \vv\in \gV,\\
    \Psi^k(\bar{\vy},\bar{\vv}) \leq \Psi^k(\vy,\bar{\vv}),\quad \forall \vy\in \Upsilon_\tau.
    \end{cases}
\end{equation*}

(ii) A point $(\bar{\vy},\bar{\vv})\in \Upsilon_\tau\times \gV$ is said to be stationary for Problem $\text{PPLA}_k$ if
\[\vzero \in \nabla \Psi^k(\bar{\vy},\bar{\vv}) + \cN_{\Upsilon_\tau}(\bar{\vy})\times \cN_\gV(\bar{\vv}).
\]
\end{definition}
The reader is referred to \cite[Theorem 10.1]{rockafellar2009} for an explanation of why partial optimality implies stationarity.

\begin{algorithm}[ht]
\caption{Alternating minimization algorithm for solving Problem $\text{PPLA}_k$}\label{alg:ama}
\vskip6pt
\begin{algorithmic}
\Procedure{AMA}{$\rho^k,\beta^k, \vy^k,\vv^k;\vy^{k,0}$}
    \For{$j = 0, 1, \ldots$}
        \State Step (i): Solve the following problem to get $\vv^{k,j+1}$:
\begin{eqnarray}\label{alg:eq1}
\displaystyle \min_{\vv} \ \Psi^k(\vy^{k,j}, \vv), \quad {\rm s.t.} \ \vv\in \gV.
\end{eqnarray}
        \State Step (ii): Solve the following problem to get $\vy^{k,j+1}$:
\begin{eqnarray}\label{alg:eq2}
\displaystyle \min_{\vy} \ \Psi^k(\vy,\vv^{k,j+1}), \quad {\rm s.t.} \ \vy\in \Upsilon_\tau.
\end{eqnarray}
\State Step (iii): If $(\vy^{k,j+1},\vv^{k,j+1})$ is a partially optimal solution of Problem $\text{PPLA}_k$, stop.
    \EndFor
    \State Return $(\vy^{k,j+1},\vv^{k,j+1})$.
\EndProcedure
\end{algorithmic}
\end{algorithm}

Below, we discuss how the two subproblems involved --- namely, Problems~\eqref{alg:eq1} and \eqref{alg:eq2} --- can be efficiently solved. To begin with, note that $\Psi^k(\vy, \vv) = \sum_{a \in \gA} \psi_a^k(\evy_a, \evv_a)$, where
\begin{equation}
\begin{split}
    \psi_a^k(\evy_a, \evv_a) = ~&t_a(\evy_a,\evv_a) \cdot \evv_a + \eta \cdot G_a(\evy_a) + \rho^k \cdot \int_0^{\evv_a} t_a(\evy_a,w) \dif w  \\
    & - \rho^k \cdot \nabla_{\evy_a} g(\vy^k) \cdot \evy_a + \rho^k \cdot \beta^k \cdot (\evy_a - \evy_a^k)^2 + \rho^k \cdot \beta^k \cdot (\evv_a - \evv_a^k)^2 + \text{constant}.
    \label{eq:psi}
\end{split}
\end{equation}

\smallskip
\begin{itemize}[leftmargin=*]
\item \textbf{Problem~\eqref{alg:eq1}} is equivalent to solving a standard UE traffic assignment problem, where the link cost function is defined as $s_a(\evy_a^{k, j},\evv_a) := \nabla_{\evv_a} \psi_a^k(\evy_a^{k, j}, \evv_a)$. By Equation \eqref{eq:psi},
\begin{equation*}
    s_a(\evy_a^{k, j},\evv_a) = (1+\rho^k) \cdot t_a(\evy_a^{k,j},\evv_a) + \evv_a \cdot \nabla_{\eva} t_a(\evy_a^{k,j},\evv_a) +  2\cdot \rho^k \cdot \beta^k \cdot (\evv_a-\evv_a^k).
\end{equation*}
As $t_a(\evy_a^{k, j},\evv_a)$ is assumed to be strictly increasing and convex in $\evv_a$, it follows that
\begin{equation}
\nabla_{\evv_a} s_a(\evy_a^{k, j},\evv_a) = (2 + \rho^k) \cdot \nabla_{\evv_a} t_a(\evy_a^{k, j},\evv_a) + \evv_a \cdot \nabla_{\evv_a}^2 t_a(\evy_a^{k, j},\evv_a)  + 2 \cdot \rho^k \cdot \beta^k > 0,
\end{equation}
which means the ``link cost" $s_a(\evy_a^{k, j},\evv_a)$ strictly increases with the link flow $\evv_a$. This property allows Problem~\eqref{alg:eq1} to be solved via a wide range of efficient route-based traffic assignment algorithms from the literature \citep[e.g.,][]{jayakrishnan1994faster, xie2018greedy, li2024wardrop}.

\item \textbf{Problem~\eqref{alg:eq2}.}  Proposition~\ref{prop:closedsolution} indicates that to solve Problem~\eqref{alg:eq2}, one first needs to solve $|\gA|$ one-dimensional problems:
\begin{equation*}
    \evy_a^* = \argmin_{0 \leq \evy_a \leq u_a} ~\psi_a^k(\evy_a, \evv_a^{k, j + 1}),
\end{equation*}
which can be easily checked to be strongly convex and can be solved efficiently using the bisection method. 
The next step is to identify the $\tau$ smallest values among $\{\psi_a^k(\evy_a^*, \evv_a^{k, j + 1}) - \psi_a^k(0, \evv_a^{k, j + 1}) : a \in \gA\}$. Letting $\bar{\gA}$ denote the corresponding index set, the update is then given by $\evy_a^{k, j + 1} = \evy_a^*$ if $a \in \bar{\gA}$, and $0$ otherwise.
\end{itemize}

\smallskip
\noindent
\textbf{Convergence.} Before proceeding, we analyze the convergence of Algorithm~\ref{alg:ama}. Note that when Algorithm~\ref{alg:ama} terminates after a finite number of iterations, the returned solution is evidently a partially optimal solution to Problem~$\text{PPLA}_k$. Therefore, we focus our analysis on the case where the algorithm generates an infinite sequence of iterates. In this setting, the following theorem establishes the convergence of the algorithm to a partially optimal solution.

\begin{theorem}\label{thm:ama}
Suppose that Algorithm~\ref{alg:ama} generates an infinite sequence $\{(\vy^{k,j}, \vv^{k,j})\}_{j=1}^\infty$. Then,
\vspace{0.3em}
\begin{itemize}
\setlength{\itemsep}{0.5em}
\item[(i)] the sequence $\{\Psi^k(\vy^{k,j}, \vv^{k,j})\}_{j=0}^\infty$ is strictly decreasing and converges to a unique limit;
\item[(ii)] the sequence $\{(\vy^{k,j}, \vv^{k,j})\}_{j=1}^\infty$ has at least one accumulation point;
\item[(iii)] any accumulation point of $\{(\vy^{k,j}, \vv^{k,j})\}_{j=1}^\infty$ is a partially optimal solution of $\text{PPLA}_k$~\eqref{approP}.
\end{itemize}
\end{theorem}
\begin{proof}
    See Appendix \ref{app:ama} for the proof.
\end{proof}

\subsubsection{Algorithm for CCBCEP-E}
\label{sec:original}

We are now ready to discuss how CCBCEP-E~\eqref{equi-p} can be efficiently solved. Observe that if we fix $\rho^k \equiv \rho$ and $\beta^k \equiv \beta$, then sequentially solving PPLA$_k$ via Algorithm~\ref{alg:ama}, that is, recursively setting
\begin{equation}\label{sequentialsolve}
(\vy^{k+1}, \vv^{k+1}) = \mathrm{AMA}(\rho, \beta; \vy^k, \vv^k; \vy^{k,0}), \quad k = 0, 1, \ldots,
\end{equation}
is expected to progressively solve PPA$_\rho$. By appropriately adjusting $\rho$ (and, if necessary, also $\beta$), one can then use solutions of PPA$_\rho$~\eqref{penalty_p} to gradually approximate CCBCEP-E~\eqref{equi-p}. However, this two-step procedure may be unnecessary. Instead, we propose an approach that sequentially solves PPLA$_k$ while adaptively updating both $\rho^k$ and $\beta^k$ in sync.

First, a core question is whether the solution to PPLA$_k$ can readily satisfy the equilibrium constraint of CCBCEP-E~\eqref{equi-p}, given that it is a penalized and approximate formulation of the latter. This is not straightforward: even for PPA$_\rho$, feasibility for CCBCEP-E~\eqref{equi-p} is not certainly guaranteed merely by gradually increasing the penalty parameter $\rho$, since neither problem is convex. This naturally raises the question of whether a sufficiently large $\rho$ can eliminate any violation of the constraint $\varphi(\vy,\vv) \leq 0$. We find that the answer is affirmative for the output of Algorithm~\ref{alg:ama}, even though it may only yield a partially optimal solution to $\mathrm{PPLA}_k$~\eqref{approP}, as formalized in the following proposition.

\begin{proposition}\label{prop:estimate}
    Denote $c_u = \diam(\gV)$ and let $b_u$ and $b_l$ be defined as 
    \begin{equation*}
        b_u=\max_{\vy\in \Upsilon_\tau, \vv \in \gV^*(\vy)} F(\vy, \vv) \quad \text{and} \quad b_l=\min_{\vy \in \Upsilon_\tau, \vv\in \gV} F(\vy, \vv).
    \end{equation*}
    Given any $\epsilon > 0$ and $\theta_u > 0$, as long as $\rho^k$ and $\beta^k$ satisfy
    \begin{equation}\label{condition}
    \rho^k \geq \frac{b_u+\theta \cdot c_u - b_l}{\epsilon} \quad \text{and} \quad \rho^k \cdot \beta^k\leq \theta_u, 
    \end{equation}
    then $(\vy^{k+1}, \vv^{k+1}) = \mathrm{AMA}(\rho, \beta; \vy^k, \vv^k; \vy^k)$, i.e., the solution generated by Algorithm \ref{alg:ama} for solving PPLA$_k$ with the initial point $\vy^{k,0} = \vy^k$, satisfies
    \[
    \Phi(\vy^{k+1}, \vv^{k+1}; \vy^k, \vv^k) \leq \epsilon,
    \]
    which readily implies that $\varphi(\vy^{k+1}, \vv^{k+1}) \leq \epsilon$ by Equation \eqref{eq:phi-Phi}.
\end{proposition}
\begin{proof}
    See Appendix \ref{app:estimate} for the proof.
\end{proof}

Although Proposition~\ref{prop:estimate} does not provide an explicit way to choose $\rho^k$ and $\beta^k$ --- since computing $b_u$ and $b_l$ is generally intractable --- it offers valuable insight: the output $(\vy^{k+1}, \vv^{k+1})$ of Algorithm~\ref{alg:ama} is guaranteed to be $\epsilon$-approximately feasible for CCBCEP-E~\eqref{equi-p}, provided that $\rho^k$ is sufficiently large and $\beta^k$ is sufficiently small. This holds despite the fact that Algorithm~\ref{alg:ama} may yield only a partially optimal solution to the intrinsically non-convex PPLA$_k$. Motivated by this observation, we propose to progressively increase $\rho^k$ and decrease $\beta^k$ in tandem while sequentially solving PPLA$_k$, which leads to the development of Algorithm~\ref{alg:main}. As this approach fundamentally exploits the DC structure of the gap function in CCBCEP-E~\eqref{equi-p} and incorporates a penalty scheme, we refer to it as the penalized DC (PDC) approach throughout this paper.

\begin{algorithm}
\caption{The PDC approach for solving CCBCEP-E~\eqref{equi-p}} \label{alg:main}
\vskip6pt
\begin{algorithmic}
\Procedure{PDC}{$\epsilon_1, \epsilon_2, \epsilon_3 > 0$, $0 < \theta_l < \theta_u$, $\rho^0 > 0$, $\sigma > 1$; $\vy^0 \in \Upsilon_{\tau}$}
    \State Set $\vv^0 \in \gV^*(\vy^0)$ and select $\beta^0 \in [\theta_l / \rho^0, \theta_u / \rho^0]$.

     \For{$k = 0, 1, \ldots$}
        \State Step (i): Set $(\vy^{k+1}, \vv^{k+1}) = \mathrm{AMA}(\rho^k, \beta^k, \vy^k, \vv^k; \vy^k)$ by running Algorithm \ref{alg:ama}.
        \State Step (ii): If $\|\vy^{k+1}-\vy^{k}\|_2 \leq \epsilon_1$, $\|\vv^{k+1}-\vv^{k}\|_2 \leq  \epsilon_2$ and $\Phi(\vy^{k+1}, \vv^{k+1}; \vy^{k}, \vv^{k}) \leq \epsilon_3$, stop.
        \State Step (iii): Update the penalty parameter as follows:
        \begin{equation}
            \rho_{k+1} =
            \begin{cases}
                \sigma \cdot \rho_{k}, \quad &\text{if}~\Phi(\vy^{k}, \vv^{k}; \vy^{k+1}, \vv^{k+1}) > \epsilon_3, \\
                \rho_{k}, \quad &\text{otherwise}.
            \end{cases}
        \end{equation}
        \State Step (iv): Select $\beta_{k+1} \in [\theta_l / \rho^{k + 1}, \theta_u / \rho^{k + 1}]$.
   \EndFor
   \State Return $(\vy^{k + 1},\vv^{k + 1})$.
\EndProcedure
\end{algorithmic}
\end{algorithm}

We explain the design of Algorithm~\ref{alg:main} as follows. The algorithm iteratively solves PPLA$_k$ using Algorithm~\ref{alg:ama} in Step~(i), while updating both the penalty parameter $\rho^k$ and the regularization parameter $\beta^k$ at each iteration in Steps~(iii) and (iv). The stopping criteria in Step~(ii) require convergence in both $\vy^k$ and $\vv^k$, as well as approximate constraint satisfaction given by $\Phi(\vy^{k+1}, \vv^{k+1}; \vy^{k}, \vv^{k}) \leq \epsilon_3$, which ensures $\varphi(\vv^{k+1}; \vy^{k+1}) \leq \epsilon$ by Equation \eqref{eq:phi-Phi}. Here, we check $\Phi(\vy^{k+1}, \vv^{k+1}; \vy^{k}, \vv^{k})$ instead of $\varphi(\vv^{k+1}; \vy^{k+1})$, since computing the latter would require solving an additional convex optimization problem. In Step~(iii), the penalty parameter is increased if approximate feasibility is not yet achieved. Step~(iv) sets $\beta^{k + 1}$ such that $\beta^{k+1} \cdot \rho^{k+1}$ always lies within $[\theta_l, \theta_u]$. The upper bound $\theta_u$ is required by Condition~\eqref{condition} to ensure the approximate feasibility of the iterates, while the lower bound $\theta_l$ guarantees a uniform sufficient decrease when solving PPLA$_k$ at each iteration. This balance is crucial for the convergence analysis of Algorithm~\ref{alg:main}, as will be discussed.

\smallskip
\noindent
\textbf{Convergence.} We close this section with a convergence analysis of Algorithm~\ref{alg:main}. Since Problem~\eqref{equi-p} is non-convex, our analysis focuses on convergence to the widely used notion of $\epsilon$-approximate Karush-Kuhn-Tucker (KKT) stationary solutions, which we formally define below.

\begin{definition}
\label{def:kkt}
(i)  We say that $(\vy^*,\vv^*)$ is a KKT stationary solution of CCBCEP-E~\eqref{equi-p} if there exists a nonnegative multiplier $\mu$ such that
    \begin{equation*}
    \begin{cases}
    (\vy^*,\vv^*)\in \Upsilon_\tau \times \gV, \quad \varphi(\vv^*; \vy^*)\leq 0,\\
    0\in \nabla F(\vy^*,\vv^*) + \nabla \varphi(\vv^*; \vy^*) \cdot \mu + \cN_{\Upsilon_\tau}(\vy^*)\times \cN_{\gV}(\vv^*).
    \end{cases}
    \end{equation*}
(ii) We say that $(\vy^*,\vv^*)$ is an $\epsilon$-approximate KKT stationary  solution of CCBCEP-E~\eqref{equi-p} if there exists a nonnegative multiplier $\mu$ such that 
    \begin{equation*}
    \begin{cases}
    (\vy^*,\vv^*)\in \Upsilon_\tau \times \gV, \quad \varphi(\vv^*; \vy^*)\leq \epsilon,\\
    \operatorname{dist}\left(0,\nabla F(\vy^*,\vv^*) + \nabla \varphi(\vv^*; \vy^*) \cdot \mu + \cN_{\Upsilon_\tau}(\vy^*)\times \cN_{\gV}(\vv^*)\right)\leq \epsilon.
    \end{cases}
    \end{equation*}
\end{definition}
By the definition, when $\epsilon = 0$, an $\epsilon$-approximate KKT stationary solution coincides with an exact KKT stationary solution. In what follows, we prove the convergence of Algorithm~\ref{alg:main} toward $\epsilon$-approximate KKT stationary solutions for arbitrarily small $\epsilon$.

As a first step in our analysis, we show that Step~(i) of Algorithm~\ref{alg:main} yields a sufficient decrease in the following merit function:
\[
\psi(\vy, \vv; \rho) = F(\vy, \vv) + \rho \cdot \varphi(\vy,\vv).
\]
Note that $\psi(\vy, \vv; \rho)$ is the objective function of PPA$_\rho$~\eqref{penalty_p}, which is a weighted sum of the total travel cost $F(\vy, \vv)$ and the gap function $\varphi(\vy,\vv)$.

\begin{proposition}\label{prop:decrease}
The iterates generated by Algorithm~\ref{alg:main} satisfy
\begin{equation}
\psi(\vy^{k+1}, \vv^{k+1}; \rho^k) \leq  \psi(\vy^{k}, \vv^{k}; \rho^k) - \rho^k \cdot \beta^k \cdot \|(\vy^{k+1} - \vy^k, \vv^{k+1} - \vv^k)\|_2^2, \quad \forall k \geq 0. 
\label{eq:suff-de}
\end{equation}
\end{proposition}
\begin{proof}
    See Appendix \ref{app:decrease} for the proof.
\end{proof}

Based on Propositions \ref{prop:estimate} and \ref{prop:decrease}, we then establish the finite termination of Algorithm~\ref{alg:main}.

\begin{proposition}\label{prop-welldefined}
In Algorithm~\ref{alg:main}, the termination condition specified in Step~(ii) will be satisfied after a finite number of iterations.
\end{proposition}
\begin{proof}
    See Appendix \ref{app:prop-welldefined} for the proof.
\end{proof}

We conclude this section by formally establishing the convergence guarantee of Algorithm~\ref{alg:main}. The following theorem shows that, for any prescribed tolerance, the algorithm computes an $\epsilon$-approximate KKT stationary solution in finitely many iterations.

\begin{theorem}\label{thm-main}
For any prescribed tolerance $\epsilon > 0$, choose $\epsilon_1$ sufficiently small and set
\begin{equation}
\epsilon_2 = \frac{\epsilon}{2\sqrt{2} \cdot \theta_u}
\quad \text{and} \quad
\epsilon_3 = \epsilon.
\label{eq:epsilon-2}
\end{equation}
Then Algorithm~\ref{alg:main} terminates after a finite number of iterations and returns an $\epsilon$-approximate KKT stationary solution of CCBCEP-E~\eqref{equi-p}.
 
\end{theorem}

\begin{proof}
    See Appendix \ref{app:thm-main} for the proof.
\end{proof}

Theorem \ref{thm-main} ensures the practical applicability of our approach, as it guarantees convergence to an $\epsilon$-approximate KKT solution within a finite number of iterations. In the next section, we demonstrate the performance of Algorithm~\ref{alg:main} through computational experiments.

\section{Numerical Study}\label{sec:num}

To examine the performance of Algorithm \ref{alg:main}, we conduct a set of numerical experiments on three networks frequently used in the transportation literature, including a small network (Hearn) and two aggregated networks of real-world cities (Sioux-Falls and Anaheim). The topology of the Hearn network, which has 9 nodes and 18 links, is shown in Figure \ref{fig:nn}. For the other two, the Sioux-Falls network has 24 nodes and 76 links, and the Anaheim network has 416 nodes and 914 links. For more details of these networks, the reader may consult the Transportation Networks GitHub Repository \citep{networks-github}. 
\begin{figure}[ht]
    \centering
    \includegraphics[width=0.6\textwidth]{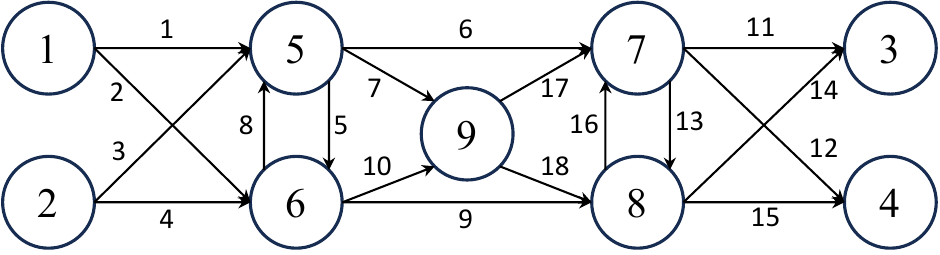} 
    \centering
    \caption{Topology of the Hearn network.}
    \label{fig:nn}
\end{figure}

In all networks, we assume the monetary cost $G_a(\evy_a)$ of expanding each link $a$ is quadratic with respect to the added capacity $\evy_a$, i.e., for some $b_a > 0$, 
\begin{equation}
    G_a(\evy_a) = b_a \cdot \evy_a^2, \quad \forall a \in \gA,
\end{equation}
Meanwhile, we assume the link cost function $t_a(\evy_a,\evv_a)$ follows the BPR equation \eqref{bpr}.

\smallskip
\noindent
\textbf{Alternative approaches.} 
Benchmarking CCBCEP \eqref{p} is nontrivial: to the best of our knowledge, there are no readily applicable baselines that optimize directly under a hard cardinality constraint. Accordingly, we compare our framework against two practical proxies:  \textbf{M1} (sensitivity-based prescreening) and \textbf{M2} ($\ell_1$ regularization). 

\smallskip
\begin{itemize}[leftmargin=*]

\item
\textbf{M1.} The prescreening approach evaluates all links with $u_a > 0$ (i.e., links eligible for expansion) by computing their marginal decrease in the objective value per unit of expansion, namely $e_a$ in Equation~\eqref{eq:sensitivity}. The $\tau$ links with the largest $e_a$ values are then selected to form $\hat{\gA}$, and the solution to the corresponding BCEP$(\hat{\gA})$ is used as a baseline for comparison with our approach. To solve BCEP$(\hat{\gA})$, we adopt the descent algorithm of \citet{yu2025scalable}, which iteratively improves $\vy$ using the gradient $\nabla F^*(\vy)$, where $F^*(\vy) = F(\vy, v^*(\vy))$ and $v^*(\vy)$ denotes the unique link flow pattern induced by $\vy$. In their method, $\nabla F^*(\vy)$ is efficiently computed through a novel technique that integrates the advantages of implicit differentiation and automatic differentiation.

\item
\textbf{M2.} %
The $\ell_1$-regularization technique has been widely used to induce fewer non-zero weights (i.e., sparsity) in machine learning and statistics \citep{tibshirani1996regression}. We introduce this regularization term to the BCEP objective for inducing sparsity in the expansion plan. Specifically, it solves
\begin{equation}\label{cndp-l1}
\begin{aligned}
\ell_1\text{-BCEP}\qquad 
\min_{\vy,\vv}\ & F(\vy,\vv) + \alpha \cdot \|\vy\|_1,\\
\text{s.t.}\ & \vy\in \gY,\quad \vv \in \gV^*(\vy),
\end{aligned}
\end{equation}
where $\alpha > 0$ is a regularization parameter that promotes sparsity in $\vy$. 
To the best of our knowledge, this formulation has not been explored in the CNDP literature. Yet, a readily available approach to solving $\ell_1$-BCEP with a fixed $\alpha$ is again to employ the descent method of \citet{yu2025scalable}, which in this case operates on the (sub)gradient 
$
\nabla F^*(\vy) + \alpha \cdot \vone.
$

Compared with the proposed CCBCEP, the $\ell_1$-regularized model provides only \emph{indirect} control over sparsity: the relationship between $\alpha$ and the support size of the solution is unclear, and for any desired $\tau$, there is no guarantee that a suitable $\alpha$ will yield a solution with exactly $\tau$ nonzeros. 
Moreover, the $\ell_1$ regularization introduces shrinkage bias on the selected links, potentially underestimating desirable expansion levels. To employ the $\ell_1$-regularization approach as a benchmark, we design a dedicated procedure that dynamically adjusts $\alpha$ to approximate the target sparsity and incorporates a post-processing step to remove shrinkage bias. The search for $\alpha$ is carried out in two phases. In the \emph{coarse search} phase, starting from an initial $\alpha > 0$ (e.g., $\alpha=0.1$), we repeatedly multiply $\alpha$ by a constant factor $\gamma_c>1$ (e.g., $\gamma_c=2$) until the number of positive components in $\vy$ does not exceed the target cardinality $\tau$, thereby ensuring sufficient regularization. In the subsequent \emph{refinement} phase, we decrease $\alpha$ in smaller multiplicative steps with factor $\gamma_r \in (0,1)$ (e.g., $\gamma_r=0.95$) until the number of nonzero components exceeds $\tau$. The last $\alpha$ satisfying the cardinality is then selected. Because the $\ell_1$ penalty shrinks nonzero entries toward zero, the final solution is debiased by re-solving an unregularized BCEP restricted to the identified support, and the resulting solution is used for comparison with Algorithm~\ref{alg:main}. Through this dedicated procedure, the $\ell_1$-based approach can yield strong benchmark solutions: if the support induced by the chosen $\alpha$ coincides with the optimal support, the debiasing step may recover the optimal CCBCEP solution. Thus, this procedure can be regarded as a powerful reference point against which to assess the performance of our proposed method.

\end{itemize}

\smallskip
\noindent
\textbf{Reference points.} To evaluate the effectiveness of the solutions returned by all tested algorithms, we compare their objective values against several reference points.

\smallskip
\begin{itemize}[leftmargin=*]
\item The total travel time at UE when no capacity is expanded, namely $F_0 = F(\vy, \vv)\mid_{\vy = \vzero, \ \vv \in \gV^*(\vy)}$, which serves as a natural upper bound.

\item The optimal objective value of the system-optimal (SO) problem, denoted by $F_{\mathrm{so}}$,
\begin{eqnarray}\label{so}
\begin{array}{rl}
\displaystyle{\mathrm{SO}}~~~~~~\min_{\vy,\vv} & ~F(\vy,\vv), \\ 
{\rm s.t.} & ~\vy\in \gY, \quad \vv \in \gV,
\end{array}
\end{eqnarray}
where both the equilibrium and cardinality constraints are removed. This corresponds to the ``least constrained'' outcome, thus offering a natural lower bound. The SO problem~\eqref{so} is a simple convex program which can be solved globally through an alternating minimization approach that, starting from an $\vy^0 \in \gY$ (e.g., $\vy^0 = \vzero$), iteratively updates, for $k = 0, 1, \ldots$,
\begin{equation*}
    \vv^{k + 1} = \argmin_{\vv \in \gV}~F(\vy^k, \vv) \quad \text{and then} \quad \vy^{k + 1} = \argmin_{\vy \in \gY}~F(\vy, \vv^{k + 1}),
\end{equation*}
where the first is a standard SO traffic assignment problem while the second is a simple box-constrained convex program.

\end{itemize}
\smallskip
For better comparison, the objective corresponding to any $\vy \in \gY$ is reported on a relative scale with respect to $F_0$ and $F_{\mathrm{so}}$, computed by
\[
\frac{F(\vy, \vv) - F_{\mathrm{so}}}{F_0 - F_{\mathrm{so}}} \Big |_{\vv \in \gV^*(\vy)}.
\]
Throughout the experiments, whenever it is necessary to solve a UE problem, we employ \citet{xie2018greedy}'s algorithm, a route-based traffic assignment algorithm with cutting-edge efficiency. All experiments are coded using Python, and all results are produced on a MacBook Air (Apple M3 Chip) with an 8-core CPU.

\subsection{Hearn network}

In the Hearn network, there are 4 OD pairs: $1 \to 3$, $1 \to 4$, $2 \to 3$, and $2 \to 4$, where the demands are 1, 2, 3, and 4, respectively. The link data, including the free-flow travel time $t_{0, a}$, the capacity $c_a$, and the parameter in the expansion cost function $b_a$, are configured according to Table \ref{tab:hearn-link}.  In the experiment, we set $\eta = 1$ and $\tau = 2$ for all links. Under this setting, it can be computed that $F_0 = 245.6$ and $F_{\text{so}} = 183.6$. To facilitate comparison across methods, all reported objectives shall be presented on a relative scale with respect to these two reference points, as discussed.
\begin{table}[htbp]
  \centering
  \caption{The values of  $t_{0, a}$, $c_a$, and $b_a$ in the Hearn network.}
  \vspace{2pt}
  \small
    \begin{tabular}{ccccccccccccccccccc}
    \toprule
    Link  & 1     & 2     & 3     & 4     & 5     & 6     & 7     & 8     & 9     & 10    & 11    & 12    & 13    & 14    & 15    & 16    & 17    & 18 \\
    \midrule
    $t_{0, a}$ & 5     & 6     & 3     & 9     & 9     & 2     & 8     & 2.67  & 6     & 7     & 3     & 6     & 2     & 8     & 6     & 4     & 4     & 8 \\
    $c_a$ & 1.2   & 1.8   & 3.5   & 3.5   & 2     & 1.1   & 2.6   & 1.1   & 3.3   & 3.2   & 2.5   & 2.4   & 1.9   & 3.9   & 4.3   & 3.6   & 2.6   & 3 \\
    $b_a$     & 5     & 18    & 9     & 9     & 18    & 2     & 24    & 5.33  & 6     & 21    & 3     & 18    & 4     & 24    & 6     & 8     & 12    & 24 \\
    \bottomrule
    \end{tabular}%
  \label{tab:hearn-link}%
\end{table}%

Given the relatively small size of the network, the global optimal solution to CCBCEP~\eqref{p} can be obtained via a brute-force enumeration of all possible combinations of expanded links. For each candidate support $\hat{\gA}$ (with $|\hat{\gA}| = 2$), fixing the nonzero locations reduces CCBCEP~\eqref{p} to a standard BCEP($\hat{\gA}$) whose decision variable is two-dimensional. 
We then solve each BCEP($\hat{\gA}$) to (near-)global optimality via a brute-force search over the two expansion levels and retain the best objective value across all supports. Using this procedure, we find that the global optimum corresponds to $y_6 = 2.554$ and $y_{11} = 0.793$, which yields an objective value of $210.1$, or $42.7\%$ on the relative scale.
When implementing Algorithm~\ref{alg:main}, we set $\rho_0 = 1$, $\sigma = 1.25$, $\theta_l = 10$, $\theta_u = 20$, $\beta_{k} = (\theta_l + \theta_u) / \rho^k$, and $\epsilon_1 = \epsilon_2 = \epsilon_3 = 10^{-3}$. All alternative approaches are implemented as described earlier. Below, we summarize the performance of all tested approaches.

\smallskip
\noindent
\textbf{M1}. The sensitivity analysis ranks Links~6 and~9 as having the highest $e_a$ values. Solving the corresponding BCEP instance yields an optimal expansion of $\evy_6 = 2.454$ and $\evy_9 = 0.315$, at which the relative objective value is $51.0\%$, higher than the global solution's $42.7\%$.

\smallskip
\noindent
\textbf{M2}. Starting with $\alpha = 0.1$, $\ell_1$-BCEP expands 11 links. Increasing $\alpha$ to $12.8$ satisfies the cardinality constraint. Decreasing $\alpha$ twice ($12.8 \to 11.55 \to 10.97$) violates the constraint again. So, we set $\alpha = 11.55$, which produces a solution expanding Links 6 and 11, identical to the global optimal solution. Hence, after debiasing, the eventual output of \textbf{M2} matches the global optimum.

\smallskip
\noindent\textbf{Algorithm~\ref{alg:main}}. The algorithm converges smoothly to a stationary point, identifying Links 6 and 11 for expansion, with expansion levels identical to those of the global optimal solution, achieving a relative objective of $42.7\%$. To show more details of the convergence trajectory, Figure \ref{fig:hearn-solution} illustrates how the support of the solution evolves over iterations: initially, the algorithm activates Links 6 and 9, but as the iterations progress, Link 9 is dropped and Link 11 enters the support, after which the expansion levels gradually stabilize. This trajectory implies that the algorithm is capable of correcting early-stage selections and steering the solution toward a globally optimal support set.

\begin{figure}[ht]
  \centering
    \includegraphics[width=0.95\textwidth]{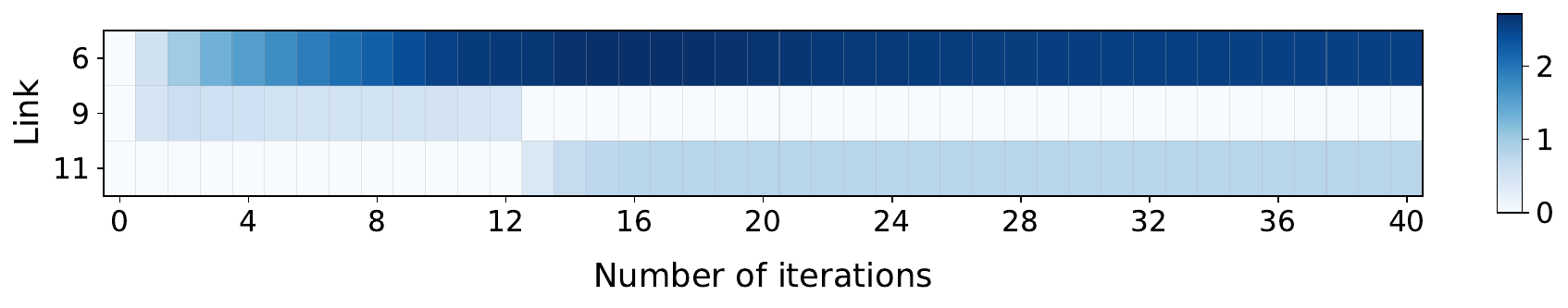} 
    \caption{Evolution of $\vy^k$ in Algorithm \ref{alg:main} in the Hearn network.}
    \label{fig:hearn-solution}
\end{figure}

In addition to tracking the decision variables, we also monitor the update rules for the penalty and regularization parameters, as well as the convergence behavior of error metrics.
Figure~\ref{fig:hearn-pdc}-(a) illustrates the trajectories of the penalty parameter $\rho^k$ and the regularization parameter $\beta^k$. We observe that $\rho^k$ continues to grow for more than 30 iterations, during which $\beta^k$ decreases in sync. Both parameters eventually stabilize just a few iterations before convergence.
Figure~\ref{fig:hearn-pdc}-(b) reports the evolution of the three error metrics:
$e_1^k = \|\vy^{k+1}-\vy^k\|_2$,
$e_2^k = \|\vv^{k+1}-\vv^k\|_2$, and
$e_3^k = \Phi(\vy^{k+1}, \vv^{k+1}; \vy^{k}, \vv^{k})$.
It can be seen that $e_2^k$ and $e_3^k$ reach the tolerance threshold $10^{-3}$ at almost the same iteration, after which $\rho^k$ stops increasing by design. Subsequently, $e_1^k$ decreases rapidly and also falls below the tolerance level, so the termination condition of Algorithm \ref{alg:main} is reached. 
\begin{figure}[ht]
\vskip 0.1in
\centering

\begin{subfigure}[b]{0.44\textwidth}
\includegraphics[width=0.85\columnwidth]{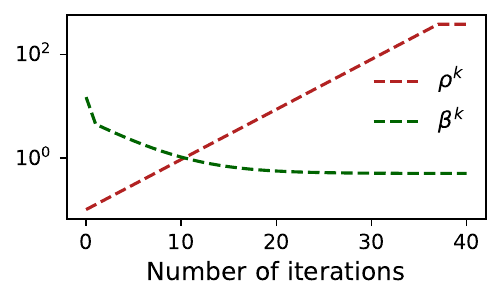}
\captionsetup{size=small}
\caption{$\rho^k$ and $\beta^k$.}
\label{fig:hearn-pdc-1}
\end{subfigure}
\begin{subfigure}[b]{0.44\textwidth}
\includegraphics[width=0.85\columnwidth]{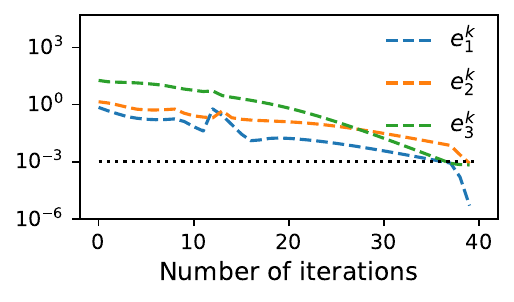}
\captionsetup{size=small}
\caption{$e_1^k$, $e_2^k$, and $e_3^k$.}
\label{fig:hearn-pdc-2}
\end{subfigure}

\caption{The trajectories of several parameters in Algorithm \ref{alg:main} on the Hearn network.}
\label{fig:hearn-pdc}
\end{figure}

\smallskip
\noindent
 \textbf{Comparison.}
In Figure~\ref{fig:hearn}, we illustrate the convergence trajectories of all tested algorithms by plotting the relative objective value of the best solution identified by each algorithm as a function of CPU time. Here, the objectives corresponding to only solutions that satisfy the cardinality constraint are evaluated. The results clearly show that $\textbf{M1}$ stucks at a suboptimal solution far away from global optimality. The best objective found by \textbf{M2} remains unchanged for a long time, as during the early coarse search phase of tuning $\alpha$, the resulting solution still does not satisfy the cardinality constraint (and hence will not be evaluated). \textbf{M2} eventually converges to global optimality, after a long period of tuning $\alpha$, and eventually a debiasing step to remove the affect of $\ell_1$ regularization.  Algorithm~\ref{alg:main}, on the contrary, converges to the global optimum straightforwardly, using much less CPU time. 

\begin{figure}[ht]
  \centering
    \includegraphics[width=0.45\textwidth]{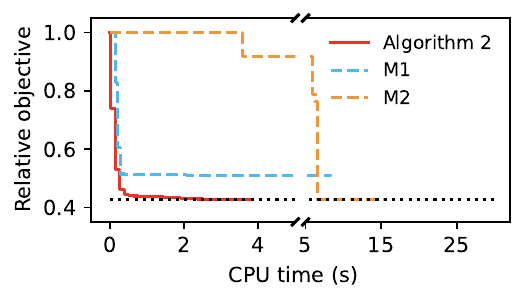} 
    \caption{Convergence trajectories of the tested algorithms on the Hearn network; the black dotted line highlights the relative objective corresponding to the global optima.}
    \label{fig:hearn}
\end{figure}
In this toy example, where the global solution can be computed exactly, the results validate the ability of Algorithm~\ref{alg:main} to solve the problem to optimality. Even in such a simple setting, however, \textbf{M1} becomes trapped in a suboptimal solution. By contrast, \textbf{M2} succeeds in finding the global optimum, albeit with higher computational cost. This suggests that \textbf{M2} can serve as a strong baseline method for comparison.

\subsection{Sioux-Falls network}

For the Sioux-Falls network experiments, our setting extends a well-established BCEP benchmark instance proposed by \citet{suwansirikul1987equilibrium} that has been used in numerous BCEP studies \citep[e.g.,][]{friesz1992simulated, marcotte1992efficient, meng2001equivalent, li2022differentiable, guo2025penalized}. The benchmark specifies the capacity expansion parameters $b_a$ for 10 designated links, denoted as $\bar \gA = \{16, 17, 19, 20, 25, 26, 29, 39, 48, 74\}$, which we maintain in our implementation. For the remaining links, we determine $b_a$ values proportionally to their lengths. The setting is shown in Table \ref{tab:sf-setting}, where links in \citet{suwansirikul1987equilibrium}'s benchmark is marked in red. As shown in the table, this configuration maintains $b_a$ values at a consistent scale across all links in the network.

\setlength{\tabcolsep}{3pt}
\begin{table}[htbp]
  \centering
  \vspace{2pt}
  \caption{The values of $b_a$ in the Sioux-Falls network.}
  \small
    \begin{tabular}{ccccccccccccccccccccccccccc}
    \toprule
    Link  & 1     & 2     & 3     & 4     & 5     & 6     & 7     & 8     & 9     & 10    & 11    & 12    & 13    & 14    & 15    & \textcolor[rgb]{ 1,  0,  0}{16} & \textcolor[rgb]{ 1,  0,  0}{17} & 18    & \textcolor[rgb]{ 1,  0,  0}{19} & \textcolor[rgb]{ 1,  0,  0}{20} & 21    & 22    & 23    & 24    & \textcolor[rgb]{ 1,  0,  0}{25} & \textcolor[rgb]{ 1,  0,  0}{26} \\
    $b_a$    & 30    & 20    & 30    & 25    & 20    & 20    & 20    & 20    & 10    & 30    & 10    & 20    & 25    & 25    & 20    & \textcolor[rgb]{ 1,  0,  0}{26} & \textcolor[rgb]{ 1,  0,  0}{40} & 10    & \textcolor[rgb]{ 1,  0,  0}{26} & \textcolor[rgb]{ 1,  0,  0}{40} & 50    & 25    & 25    & 50    & \textcolor[rgb]{ 1,  0,  0}{25} & \textcolor[rgb]{ 1,  0,  0}{25} \\
    \midrule
          &       &       &       &       &       &       &       &       &       &       &       &       &       &       &       &       &       &       &       &       &       &       &       &       &       &  \\
    \midrule
    Link  & 27    & 28    & \textcolor[rgb]{ 1,  0,  0}{29} & 30    & 31    & 32    & 33    & 34    & 35    & 36    & 37    & 38    & \textcolor[rgb]{ 1,  0,  0}{39} & 40    & 41    & 42    & 43    & 44    & 45    & 46    & 47    & \textcolor[rgb]{ 1,  0,  0}{48} & 49    & 50    & 51    & 52 \\
    $b_a$    & 25    & 30    & \textcolor[rgb]{ 1,  0,  0}{48} & 40    & 30    & 25    & 30    & 20    & 20    & 30    & 15    & 15    & \textcolor[rgb]{ 1,  0,  0}{34} & 20    & 25    & 20    & 30    & 25    & 15    & 15    & 25    & \textcolor[rgb]{ 1,  0,  0}{48} & 10    & 15    & 40    & 10 \\
    \midrule
          &       &       &       &       &       &       &       &       &       &       &       &       &       &       &       &       &       &       &       &       &       &       &       &       &       &  \\
    \midrule
    Link  & 53    & 54    & 55    & 56    & 57    & 58    & 59    & 60    & 61    & 62    & 63    & 64    & 65    & 66    & 67    & 68    & 69    & 70    & 71    & 72    & 73    & \textcolor[rgb]{ 1,  0,  0}{74} & 75    & 76    &       &  \\
    $b_a$    & 10    & 10    & 15    & 20    & 15    & 10    & 20    & 20    & 20    & 30    & 25    & 30    & 10    & 15    & 15    & 25    & 10    & 20    & 20    & 20    & 10    & \textcolor[rgb]{ 1,  0,  0}{34} & 15    & 10    &       &  \\
    \bottomrule
    \end{tabular}%
  \label{tab:sf-setting}%
\end{table}%

Throughout the experiment, we set $\tau = 10$ and $\eta = 0.001$ for all links. Under this setting, the two reference points can be computed as $F_0 = 99.94$ and $F_{\text{so}} = 60.86$. Additionally, it is easy to see the optimal solution to CCBCEP \eqref{p} is bounded above by the known optimal value from \citet{suwansirikul1987equilibrium}'s benchmark instance --- 79.90, or 48.7\% at the relative scale \citep{li2022differentiable} --- as the instance solution is feasible to CCBCEP \eqref{p}. Given that the Sioux-Falls network contains 76 links, a brute-force search with $\tau = 10$ would require evaluating approximately $10^{12}$ combinations, making it computationally infeasible. Hence, we cannot offer the global optimal solution for comparison. For the tested algorithms, when implementing Algorithm \ref{alg:main}, we set $\rho_0 = 1$, $\theta_l = 1$, $\theta_u = 2$, $\sigma = 1.05$, and $\beta_{k} = (\theta_l + \theta_u) / \rho^k$; all alternative approaches are implemented as described earlier.
\begin{figure}[ht]
  \centering
    \includegraphics[width=0.45\textwidth]{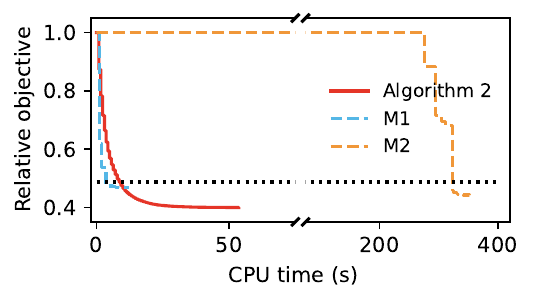} 
    \caption{Convergence trajectories of the tested algorithms on the Sioux-Falls network; the black dotted line indicates the relative objective of the best-known solution to \citet{suwansirikul1987equilibrium}'s instance, which serves as an upper bound for the new problem considered here.}
    \label{fig:sf}
\end{figure}
In Figure~\ref{fig:sf}, we plot the relative objective value of the best solution identified by each algorithm against the CPU time; as before, only solutions that satisfy the cardinality constraint at each stage of the run are evaluated. It can be checked that \textbf{M1} converges the fastest in terms of CPU time since it is equivalent to solving only one BCEP instance. However, this efficiency comes at the cost of accuracy: the resulting solution is the worst among all tested methods. \textbf{M2}, on the other hand, exhibits the slowest convergence, much longer than Algorithm \ref{alg:main}. A significant portion of its computation time is spent in the coarse search phase of $\alpha$, during which $\alpha$ progressively increases until the solution becomes sparse enough to satisfy the cardinality constraint for the first time. While \textbf{M2}'s final solution is an improvement over \textbf{M1}'s, it still underperforms against Algorithm~\ref{alg:main}.

\smallskip
To further assess the robustness of Algorithm \ref{alg:main}, we conduct additional experiments. 

\smallskip
\noindent
\textbf{Algorithmic parameter robustness.}
We first conduct two complementary sensitivity experiments on the Sioux–Falls network to assess the robustness of Algorithm~\ref{alg:main} to its internal parameters: perturb the initial penalty $\rho^{0}\in\{0.01, 0.1, 1, 10\}$;  perturb the penalty–update factor $\sigma \in \{1.01, 1.05, 1.25, 2\}$.
The corresponding convergence trajectories of the relative objective against CPU time are shown in Figure~\ref{fig:sf-pdc}-(a) and -(b). 
As shown in Figure~\ref{fig:sf-pdc}-(a), the choices $\rho^{0}=0.01, 0.1, 1$ yield very similar convergence behavior, both in speed and final objective. When $\rho^{0}$ is increased to $10$, however, progress slows and the algorithm settles at a worse relative objective. This phenomenon aligns with the theory of penalty methods in optimization, where an excessively large penalty parameter renders the problem more degenerate and consequently degrades computational performance and solution quality.
Figure~\ref{fig:sf-pdc}-(b) indicates that moderate update factors ($\sigma=1.01$ and $1.05$) drive a faster reduction of the relative objective for a given CPU budget. Increasing $\sigma$ to $1.25$ requires many more outer iterations to reach the same accuracy, while an aggressive choice ($\sigma=2$) converges much more slowly and attains a worse final objective. Overall, these tests suggest that Algorithm~\ref{alg:main} is robust to a broad range of $\rho^{0}$ and $\sigma$ values; only too large $\rho^{0}$ and $\sigma$ shall slow convergence and degrade solution quality.

\begin{figure}[ht]
\vskip 0.1in
\centering
\begin{subfigure}[b]{0.44\textwidth}
\includegraphics[width=0.85\columnwidth]{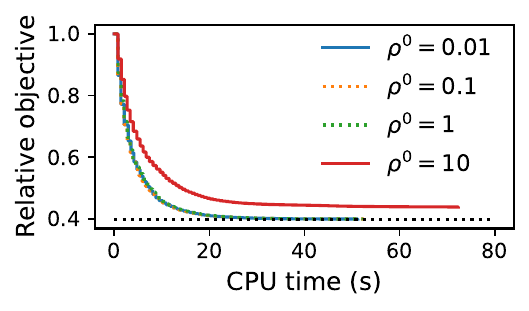}
\captionsetup{size=small}
\caption{{Perturbing $\rho^0$.}}
\label{fig:sf-s1}
\end{subfigure}
\begin{subfigure}[b]{0.44\textwidth}
\includegraphics[width=0.85\columnwidth]{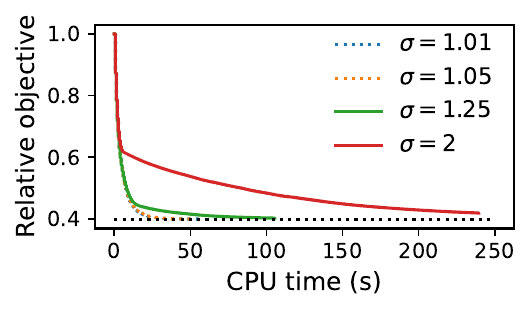}
\captionsetup{size=small}
\caption{Perturbing $\sigma$.}
\label{fig:sf-s2}
\end{subfigure}
\caption{Convergence trajectories of Algorithm \ref{alg:main} on the Sioux–Falls network under parameter perturbations; the dashed line marks the relative objective achieved by Algorithm \ref{alg:main} in the original setting.}
\label{fig:sf-pdc}
\end{figure}

\smallskip
\noindent
\textbf{Network parameter robustness.} We then generate additional scenarios by scaling $b_a$ for all links in $\gA \setminus \bar{\gA}$ by factors $\xi = 2, 3, 4, 5$. Table \ref{tab:sf-result} reports the performance of all tested algorithms, including the relative objective values of their returned solutions and the total CPU time required for convergence; results for $\xi = 1$ (the original setting) are also shown for reference. As observed, \textbf{M1} is again very fast as it is a heuristic, requiring solving only one classic BCEP; but Algorithm \ref{alg:main} requires slightly more CPU time than \textbf{M1} --- their times are still of the same order --- while being substantially faster than \textbf{M2}. More importantly, Algorithm \ref{alg:main} consistently delivers markedly better solutions than both \textbf{M1} and \textbf{M2} across all scenarios.

\setlength{\tabcolsep}{2.5pt}
\begin{table}[htbp]
  \centering
  \caption{Performance of \textbf{M1}, \textbf{M2}, and Algorithm \ref{alg:main} under different $\xi$ in the Sioux-Falls network.}
  \small
  \vspace{2pt}
    \begin{tabular}{crccrccrccrccrcc}
    \toprule
    \multirow{2}[4]{*}{Method} &       & \multicolumn{2}{c}{$\xi = 1$} &       & \multicolumn{2}{c}{$\xi = 2$} &       & \multicolumn{2}{c}{$\xi = 3$} &       & \multicolumn{2}{c}{$\xi = 4$} &       & \multicolumn{2}{c}{$\xi = 5$} \\
\cmidrule{3-4}\cmidrule{6-7}\cmidrule{9-10}\cmidrule{12-13}\cmidrule{15-16}          &       & objective & time (s) &       & objective & time (s) &       & objective & time (s) &       & objective & time (s) &       & objective & time (s) \\
    \midrule
    \textbf{M1}    &       & 46.9\% & 12    &       & 43.3\% & 18    &       & 40.0\% & 13    &       & 37.6\% & 23    &       & 35.7\% & 17 \\
    \textbf{M2}    &       & 44.3\% & 351   &       & 40.5\% & 201   &       & 39.6\% & 220   &       & 38.1\% & 234   &       & 36.3\% & 180 \\
    Algorithm \ref{alg:main}     &       & 39.9\% & 54    &       & 36.8\% & 35    &       & 34.8\% & 31    &       & 33.1\% & 66    &       & 31.9\% & 59 \\
    \bottomrule
    \end{tabular}%
  \label{tab:sf-result}%
\end{table}%

\subsection{Anaheim network}

In the final experiment, we turn to the Anaheim network. Having validated the effectiveness of Algorithm \ref{alg:main} against the two alternatives in the previous experiments, we now focus exclusively on Algorithm \ref{alg:main} to illustrate how it can be leveraged in a practical context.

The Anaheim network includes segments of an interstate highway (I-5) and several state highways (SR-22, SR-55, SR-57, and SR-91). In Figure \ref{fig:ana-highway}, these highways are highlighted in distinct colors, with the remaining urban roads shown in gray.
For all highway links, we set $b_a$ based on publicly available sources \citep{octa_i5_2025, caltrans_sr55_2022, dya_sr22_2025, usdotsr91_2025, ctc_0p670_2018}, and we assign higher values to connector roads and interchange ramps than to regular segments. In contrast, $b_a$ values for urban road links are set substantially higher than those for all highway links, reflecting the effects of higher land costs and more complex construction conditions. To further increase the necessity for capacity expansions, the travel demand for each OD pair is increased by four times relative to the original configuration.

\begin{figure}[ht]
    \centering
    \vspace{-10pt}
    \includegraphics[width=0.7\textwidth]{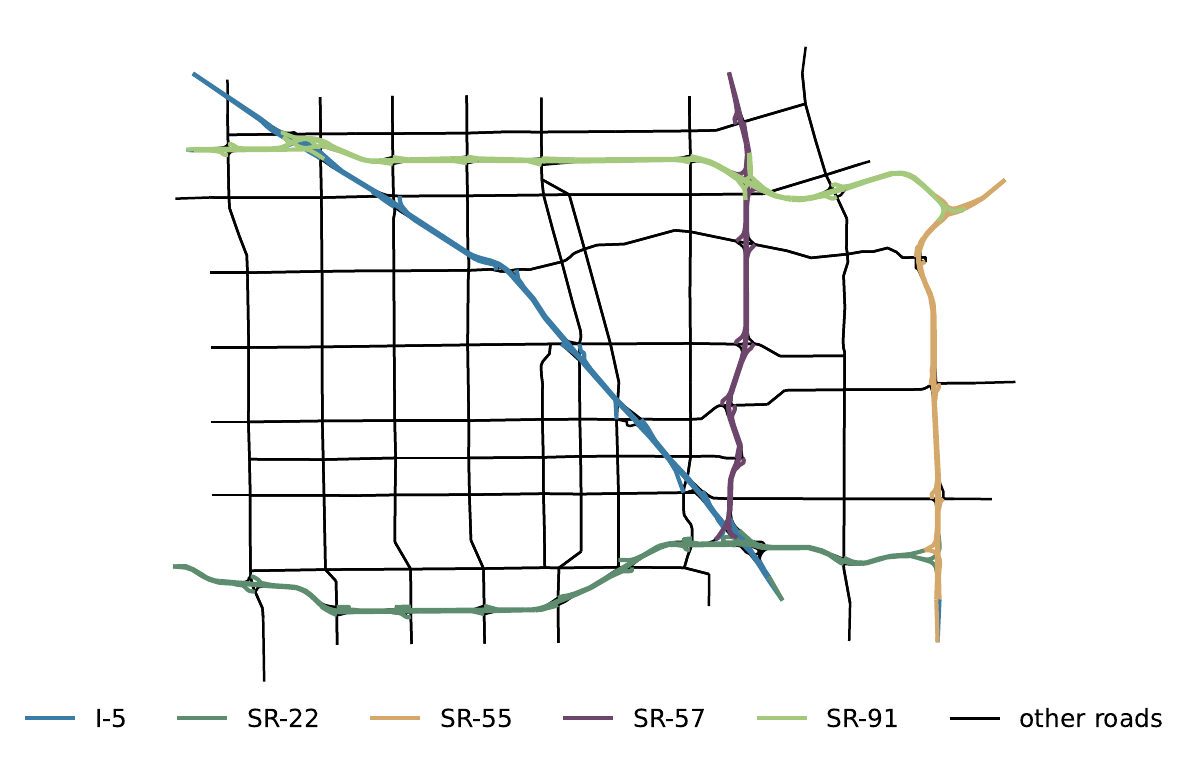} 
    \centering
    \vspace{5pt}
    \caption{An illustration of the Anaheim network.}
    \label{fig:ana-highway}
\end{figure}

Under this setting, we obtain $F_0 = 45.43$ and $F_{\text{so}} = 35.84$. We run Algorithm \ref{alg:main} for $\tau = 10, 15, 20, 25, 30, 35$. In Table \ref{tab:ana}, we report the relative objective of the solution achieved by Algorithm \ref{alg:main} given different $\tau$.
\begin{table}[htbp]
  \centering
  \caption{The relative objective corresponding to the solution achieved by Algorithm \ref{alg:main} given different $\tau$.}
  \vspace{5pt}
  \small
    \begin{tabular}{ccccccccc}
    \toprule
    $\tau$   & 5     & 10    & 15    & 20    & 25    & 30    & 35    & 40 \\
    \midrule
    relative objective & 33.5\% & 26.4\% & 22.4\% & 19.6\% & 16.9\% & 16.9\% & 15.8\% & 14.8\% \\
    \bottomrule
    \end{tabular}%
  \label{tab:ana}%
\end{table}%

The results indicate that increasing $\tau$ from 10 to 25 leads to a rapid decrease in the relative objective. However, when $\tau$ increases from 25 to 35, the improvement becomes marginal. At $\tau = 25$, the relative objective reaches $16.92\%$, which suggests that $\tau = 25$ offers a desirable balance between efficiency and the number of links expanded.  Figure \ref{fig:ana-result} further shows the expanded links identified by the algorithm for each $\tau$ (highlighted in red).
It indicates that the expanded links given $\tau = 25$, compared with those corresponding to larger values, exhibit strong spatial clustering, being concentrated in three regions: northeast, southeast, and southwest. Such spatial grouping enhances the practical feasibility of implementing the expansion plan, as construction efforts can be geographically focused.

\begin{figure}[ht]
    \captionsetup[subfigure]{font=small, skip=0pt} %
    \setlength{\belowcaptionskip}{-10pt} %
    \centering
    
    \begin{subfigure}[t]{0.3\textwidth} %
        \includegraphics[width=\linewidth]{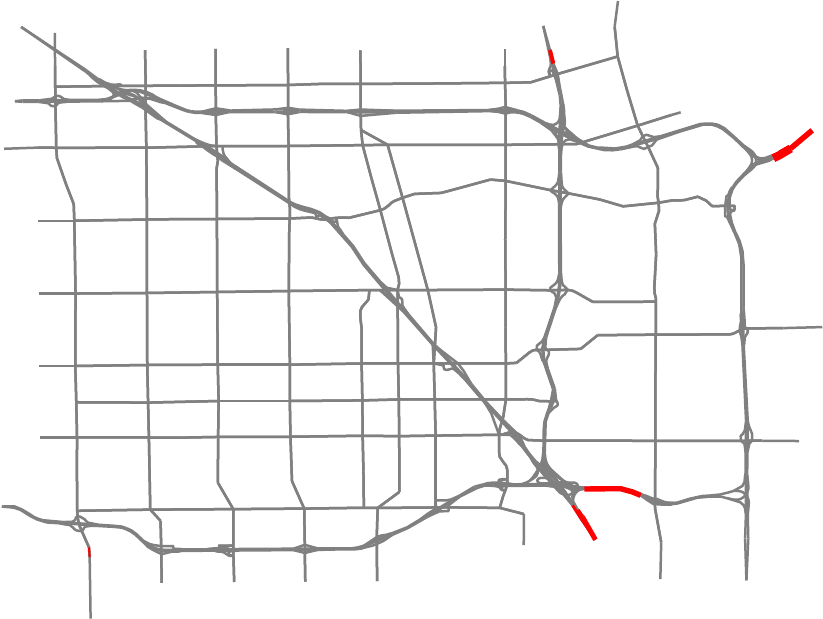}
        \caption{$\tau = 10$}
        \label{fig:10}
    \end{subfigure}
    \hfill
    \begin{subfigure}[t]{0.3\textwidth}
        \includegraphics[width=\linewidth]{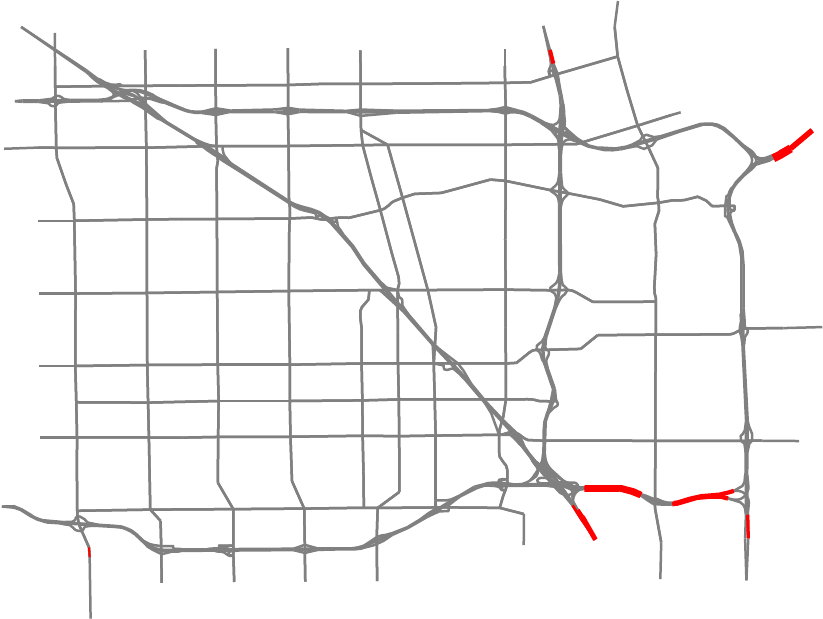}
        \caption{$\tau = 15$}
        \label{fig:15}
    \end{subfigure}
    \hfill
    \begin{subfigure}[t]{0.3\textwidth}
        \includegraphics[width=\linewidth]{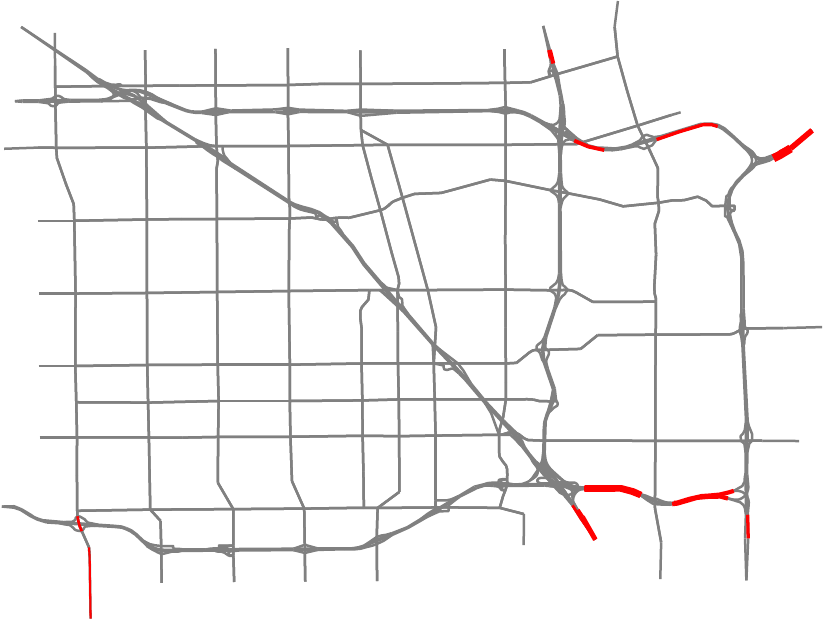}
        \caption{$\tau = 20$}
        \label{fig:20}
    \end{subfigure}
    
    \vspace{20pt} %
    
    \begin{subfigure}[t]{0.3\textwidth}
        \includegraphics[width=\linewidth]{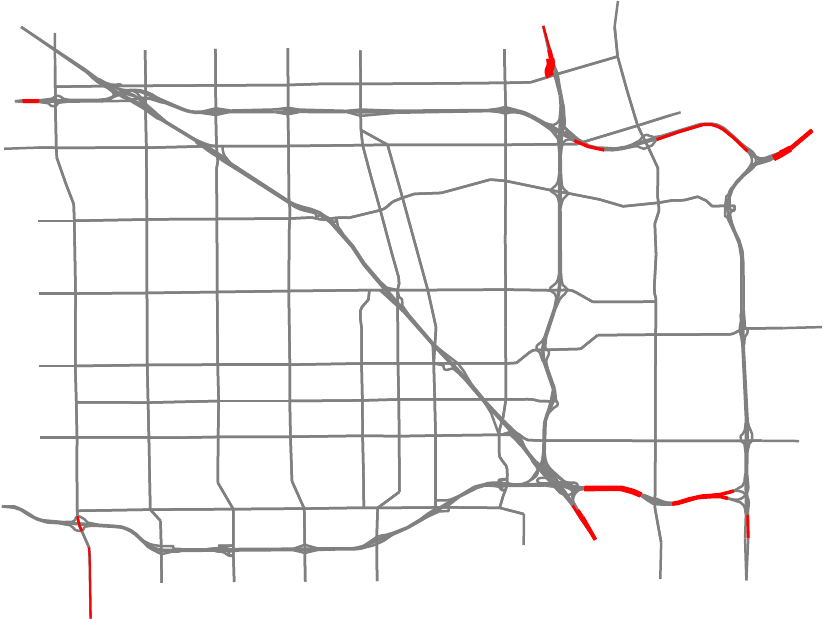}
        \caption{$\tau = 25$}
        \label{fig:25}
    \end{subfigure}
    \hfill
    \begin{subfigure}[t]{0.3\textwidth}
        \includegraphics[width=\linewidth]{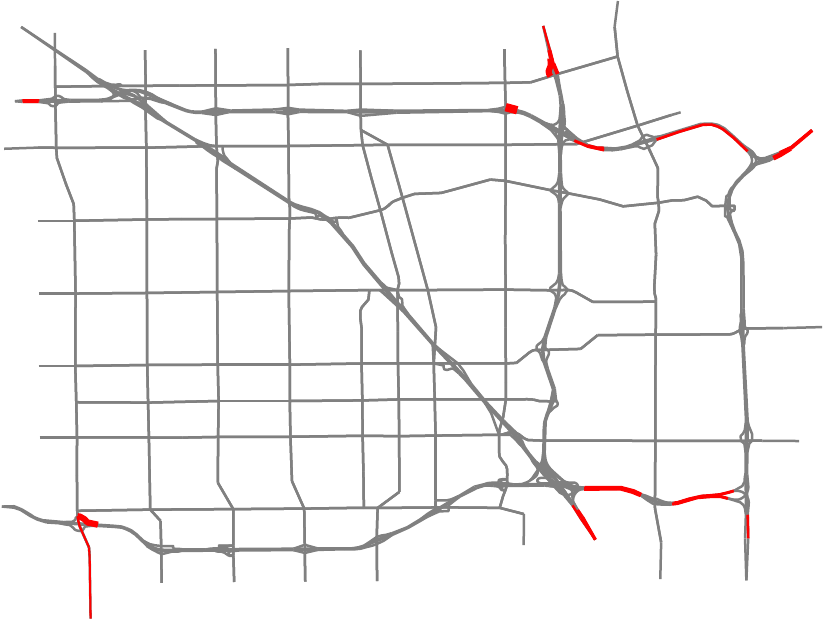}
        \caption{$\tau = 30$}
        \label{fig:30}
    \end{subfigure}
    \hfill
    \begin{subfigure}[t]{0.3\textwidth}
        \includegraphics[width=\linewidth]{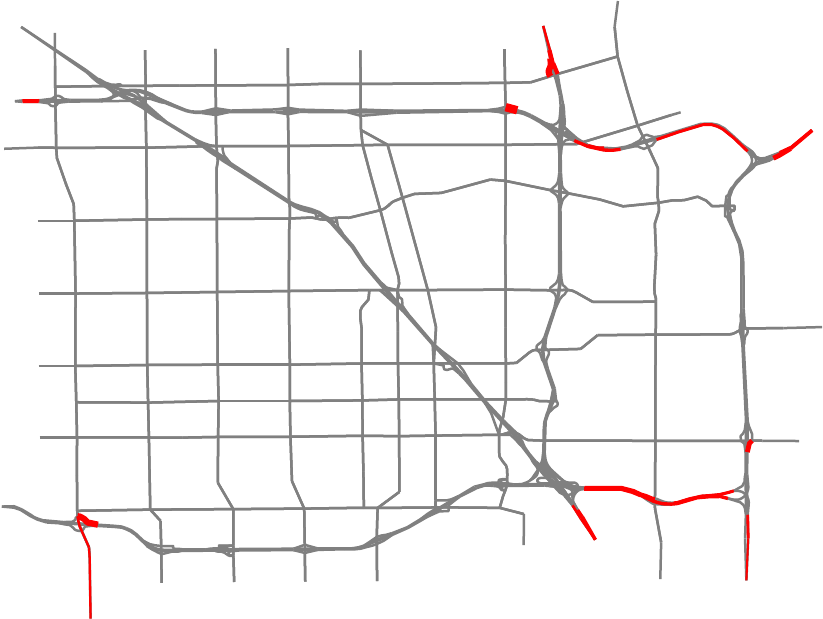}
        \caption{$\tau = 35$}
        \label{fig:35}
    \end{subfigure}

    \vspace{20pt}
    \caption{An illustration of the expanded links identified by Algorithm \ref{alg:main} given different maximum number of expanded links.}
    \label{fig:ana-result}
\end{figure}

\section{Conclusion}\label{sec:Conclusion}

As one of the most important problems in transportation and operations research, the bilevel capacity expansion problem (BCEP) has been extensively studied by researchers and practitioners for decades. In practical applications, a BCEP is often formulated under the assumption that a small set of links has been pre-determined for capacity expansion, so that the optimization focuses solely on determining the expansion levels for these links. However, this two-phase procedure --- first selecting the links to be expanded, then optimizing their capacities --- may lead to suboptimal solutions because it neglects the inherent interdependence between link selection and capacity expansion decisions. To address this challenge, our contributions can be summarized as follows.
\smallskip
\begin{itemize}[leftmargin=*]
    \item We introduced the \emph{cardinality-constrained} BCEP (CCBCEP) model, which incorporates the link selection decision directly into the bilevel formulation via an explicit cardinality constraint. This integrated approach enables simultaneous optimization of both the set of expanded links and their corresponding expansion levels, thereby capturing their interdependencies.

    \item We developed a tailored algorithmic framework to tackle the CCBCEP, which simultaneously confronts the discontinuity and non-convexity introduced by the cardinality constraint and the non-convexity arising from its bilevel structure --- a combination that cannot be directly addressed by any existing optimization toolbox.  {The core idea is to reformulate the CCBCEP as a sequence of partially penalized and linearized subproblems by exploiting the inherent DC structure of the formulation and the specific properties of the cardinality constraint set, while adaptively adjusting the penalty and regularization parameters to guarantee convergence.}

    \item We established the convergence of the proposed algorithm to an approximate Karush–Kuhn–Tucker (KKT) point of the CCBCEP with arbitrary prescribed accuracy. Notably, the convergence guarantee holds under standard regularity conditions and mild requirements on the hyperparameters, which can be readily satisfied in practice.
   
    \item We conducted a series of numerical experiments set up on practical networks to evaluate the performance of the proposed method against several alternative approaches, including sensitivity-based prescreening and $\ell_1$-regularization. The results demonstrate that our algorithm often outperforms these competing methods in terms of both solution quality and computational efficiency. In addition, the experiments show that, by explicitly incorporating a cardinality constraint, the proposed framework produces expansion plans that are not only effective in mitigating congestion but also practically implementable, as the number of links to be expanded is controlled.
    
\end{itemize}

\smallskip
The methodological and computational advances presented in this work open several promising research directions. {First, as a novel formulation in bilevel programming, the proposed CCBCEP model thus motivates considerable future research into its theoretical properties and algorithmic solutions.} Second, the framework can be extended to handle stochastic or robust variants of the capacity expansion problem, where demand and travel time functions are subject to uncertainty \citep{chen2011transport}. Third, we aim to adapt the framework to applications involving the design of dedicated facilities for connected and automated vehicles (CAVs), which has received considerable attention in recent years \citep{chen2016optimal,chen2017optimal,bahrami2020optimal}. Such extensions, however, pose additional challenges. For instance, our proposed algorithm leverages certain problem structures, such as the ability to reformulate the equilibrium constraint via a gap function with a DC structure. In the absence of such a structure, the algorithm may no longer be applicable.

\bibliographystyle{apalike} 
\bibliography{references}

\newpage
\appendix
\section{Omitted Proofs}

\subsection{Proof of Proposition \ref{prop-cardiset}}
\label{app:prop-cardiset}

First, following the definition of $\Xi_\tau$, given any $\gI \in \Xi_{\tau}$, a $\vy \in \gY$ with $\evy_a = 0$ ($\forall a \in \gA \setminus \gI$) would at most have $\tau$ positive elements. Thus, we have
\[
\bigcup_{\gI\in \Xi_\tau}\left\{\vy\in \gY: \evy_a=0, \ a \in \gA \setminus \gI\right\} \subseteq \Upsilon_\tau.
\]
Meanwhile, for any $\bar{\vy} \in \Upsilon_\tau$, it follows that $\bar{\vy} \in \gY$ and $\supp(\bar{\vy}) \leq \tau$. Letting $\gJ = \{i: \bar{\evy}_i \neq 0\}$, it then holds $|\gJ|\leq \tau$. Choosing a subset $\gI \subseteq \gA$ with $|\gI| = \tau$ such that $\gJ \subseteq \gI$, it then holds $\bar{\vy} \in \{\vy:\evy_a=0, \  \ a\in \gA\setminus \gI\}$, which indicates that
\[
\Upsilon_\tau \subseteq \bigcup_{\gI\in \Xi_\tau}\left\{\vy\in \gY: \evy_a=0, \ a\in \gA\setminus \gI\right\}. 
\]
Combining both then concludes the proof. \hfill\QED

\subsection{Proof of Proposition \ref{prop:closedsolution}} 
\label{app:closedsolution}

Without loss of generality, we assume that $\evy_a^*>0$ for all $a\in \gA$. By the definition of $\bar{\vy}^*$, it follows that $\bar{\vy}^*\in \Upsilon_\tau$. To further prove that it is an optimal solution to Problem \eqref{prop-opt}, it suffices to show that 
\begin{equation}\label{ineq:suff}
\sum_{a\in \gA} \phi_a(\bar{y}^*_a) \leq \sum_{a\in \gA} \phi_a(y_a), \quad \forall \vy\in \Upsilon_\tau.
\end{equation}
For any $\evy = (\evy_a)_{a \in \gA} \in \Upsilon_\tau$, let us denote
$
\gI = \{a \in \gA: \evy_a>0\} \quad \text{and} \quad \gI^c=\{a\in \gA: \evy_a=0\}.
$
Writing $\bar{\gA}^c = \gA \setminus \bar{\gA}$, we can then divide $\gA$ into
$
\gA = (\bar{\gA}\cap \gI) \cup  (\bar{\gA}\cap \gI^c) \cup  (\bar{\gA}^c\cap \gI) \cup (\bar{\gA}^c\cap \gI^c).
$
First, for all $a\in \bar{\gA}=(\bar{\gA}\cap \gI) \cup  (\bar{\gA}\cap \gI^c)$, it follows from the optimality of $y_a^*$ that
\begin{equation}\label{inequ1}
\phi_a(\bar{y}^*_a) = \phi_a({y}^*_a) \leq \phi_a(y_a).
\end{equation}
Second, for all $a\in \bar{\gA}^c\cap \gI^c$, it follows that $\bar{\evy}^*_a=\evy_a=0$ and thus
\begin{equation}\label{inequ2}
\phi_a(\bar{\evy}^*_a)  = \phi_a(0) = \phi_a(\evy_a).
\end{equation}
Last, by the definitions of $\gI^c$ and $\bar{\gA}$, we have
\begin{eqnarray*}
     |\gA| - \tau =  |\bar{\gA}^c| = |\bar{\gA}^c\cap \gI| + |\bar{\gA}^c\cap \gI^c| \quad \text{and} \quad |\gA|-\tau \leq |\gI^c| = |\bar{\gA}^c\cap \gI^c| + |\bar{\gA}\cap \gI^c|,
\end{eqnarray*}
which further implies that $|\bar{\gA}\cap \gI^c| \geq |\bar{\gA}^c\cap \gI|$. For all $a\in \bar{\gA}^c\cap \gI$ and $a'\in \bar{\gA}\cap \gI^c$, we then have
\begin{equation}
\begin{split}
    &(\phi_a(\bar{\evy}^*_a) - \phi_a({\evy}_a)) + (\phi_{a'}(\bar{\evy}^*_{a'}) - \phi_{a'}({\evy}_{a'}))
    = (\phi_a(0) - \phi_a({\evy}_a)) + (\phi_{a'}({\evy}^*_{a'}) - \phi_{a'}(0))\\
    &\qquad \leq  (\phi_a(0) - \phi_a({\evy}_a)) +  (\phi_{a}({\evy}^*_{a}) - \phi_{a}(0))
    = \phi_{a}({\evy}^*_{a})- \phi_a({\evy}_a))
    \leq 0,\label{inequ3}
\end{split}
\end{equation}
where the first equality follows from the definitions of $\gI^c$ and $\bar{\vy}^*$, the first inequality follows from the definition of $\bar{\gA}$, and the last inequality follows from the optimality of ${\evy}^*_{a}$. Equation\eqref{inequ3} then implies 
\begin{equation}\label{inequ4}
\sum_{a\in (\bar{\gA}^c\cap \gI) \cup (\bar{\gA}\cap \gI^c)} \phi_a(\bar{\evy}^*_a) \leq \sum_{a\in (\bar{\gA}^c\cap \gI) \cup (\bar{\gA}\cap \gI^c)}  \phi_a(\evy_a).
\end{equation}
Eventually, combining Equations \eqref{inequ1}, \eqref{inequ2}, and \eqref{inequ4} then yields Equation \eqref{ineq:suff}. \hfill\QED

\subsection{Proof of Theorem \ref{thm:ama}} 
\label{app:ama}

(i) 
To begin, note that by the design of Algorithm~\ref{alg:ama}, the following inequalities hold for all $j \geq 0$:
\begin{eqnarray}
&& \Psi^k(\vy^{k, j},\vv^{k, j + 1}) \leq \Psi^k(\vy^{k, j}, \vv),\quad \forall \vv\in \gV,\label{th:eq1}\\
&& \Psi^k(\vy^{k, j + 1},\vv^{k, j + 1}) \leq \Psi^k(\vy,\vv^{k, j + 1}),\quad \forall \vy\in \Upsilon_\tau.\label{th:eq2}
\end{eqnarray}
Since $\vy^{k, j} \in \Upsilon_\tau$ and $\vv^{k, j} \in \gV$, the above two inequalities imply that
\begin{equation}\label{th:eq3}
\Psi^k(\vy^{k, j + 1}, \vv^{k, j + 1}) \leq \Psi^k(\vy^{k, j}, \vv^{k, j + 1}) \leq \Psi^k(\vy^{k, j}, \vv^{k, j}), \quad \forall j \geq 0.
\end{equation}
We now show that $\Psi^k(\vy^{k, j + 1}, \vv^{k, j + 1}) < \Psi^k(\vy^{k, j}, \vv^{k, j})$ must hold. 

Suppose, to the contrary, that $\Psi^k(\vy^{k, j + 1}, \vv^{k, j + 1}) = \Psi^k(\vy^{k, j}, \vv^{k, j})$. Then, by \eqref{th:eq3}, we have $$\Psi^k(\vy^{k, j + 1}, \vv^{k, j + 1}) = \Psi^k(\vy^{k, j}, \vv^{k, j + 1}) = \Psi^k(\vy^{k, j}, \vv^{k, j}).$$
Substituting $\Psi^k(\vy^{k, j}, \vv^{k, j + 1})$ with $\Psi^k(\vy^{k, j}, \vv^{k, j})$ in \eqref{th:eq1} yields
\begin{equation}\label{th:eq5}
\Psi^k(\vy^{k, j}, \vv^{k, j}) \leq \Psi^k(\vy^{k, j}, \vv), \quad \forall \vv \in \gV,
\end{equation}
which indicates that $\vv^{k, j}$ is also an optimal solution of Problem~\eqref{alg:eq1}. Since Problem~\eqref{alg:eq1} has a unique solution --- thanks to the strong convexity of its objective function and the convexity of the constraint set $\gV$ --- it follows that $\vv^{k, j + 1} = \vv^{k, j}$. Substituting $\Psi^k(\vy^{k, j + 1}, \vv^{k, j + 1})$ with $\Psi^k(\vy^{k, j}, \vv^{k, j})$ in \eqref{th:eq2} and using $\vv^{k, j + 1} = \vv^{k, j}$, we obtain
\[
\Psi^k(\vy^{k, j},\vv^{k, j}) \leq \Psi^k(\vy, \vv^{k, j}),\quad \forall \vy\in \Upsilon_\tau
\]
Together with Equation~\eqref{th:eq5}, this shows that $(\vy^{k, j}, \vv^{k, j})$ is a partially optimal solution. By the termination criterion in Step (iii) of Algorithm~\ref{alg:ama}, $(\vy^{k, j}, \vv^{k, j})$ would then be the returned solution and the algorithm would terminate. This contradicts the assumption that ${(\vy^{k, j}, \vv^{k, j})}$ is an infinite sequence. Therefore, the sequence $\{\Psi^k(\vy^{k, j}, \vv^{k, j})\}_{j=1}^\infty$ is strictly monotonically decreasing. Moreover, as it is bounded below due to the compactness of both $\Upsilon_\tau$ and $\gV$, it converges to a unique limit.

(ii) Since both $\Upsilon_\tau$ and $\gV$ are compact sets, the iteration sequence ${(\vy^{k, j}, \vv^{k, j})}_{j=1}^\infty$ is bounded and therefore admits a convergent subsequence by the Bolzano–Weierstrass theorem. Consequently, at least one accumulation point exists.

(iii) Assume that $(\vy^*,\vv^*)$ is an accumulation point and $\gJ \subseteq \sN$ is the index set such that $\{(\vy^{k, j}, \vv^{k, j})\}_{j\in \gJ}$ converges to $(\vy^*,\vv^*)$ as $j\to \infty$. By the compactness of both $\Upsilon_\tau$ and $\gV$, the sequence $\{(\vy^{k, j + 1}, \vv^{k, j + 1})\}_{j \in \gJ}$ is bounded, which implies that it contains a convergent subsequence. Let $\gJ_1 \subseteq \gJ$ denote the index set such that the subsequence $\{(\vy^{k, j + 1}, \vv^{k, j + 1})\}_{j \in J_1}$ converges to $(\bar{\vy}, \bar{\vv})$. By the relationship between $(\vy^{k, j + 1}, \vv^{k, j + 1})$ and $(\vy^{k, j}, \vv^{k, j})$, it follows that for all $j \in \gJ_1$,
\begin{eqnarray*}
&& \Psi^k(\vy^{k, j}, \vv^{k, j + 1}) \leq \Psi^k(\vy^{k, j}, \vv),\quad \forall \vv\in \gV,\\
&& \Psi^k(\vy^{k, j + 1}, \vv^{k, j + 1}) \leq \Psi^k(\vy, \vv^{k, j + 1}),\quad \forall \vy\in \Upsilon_\tau.
\end{eqnarray*}
Taking limits as $j \to \infty$ along the subsequence $\gJ_1$, we obtain
\begin{eqnarray*}
&& \Psi^k(\vy^*, \bar{\vv}) \leq \Psi^k(\vy^*, \vv),\quad \forall \vv \in \gV,\\
&& \Psi^k(\bar{\vy},\bar{\vv}) \leq \Psi^k(\vy,\bar{\vv}),\quad \forall \vy \in \Upsilon_\tau.
\end{eqnarray*}
Since $\vy^* \in \Upsilon_\tau$ and $\vv^* \in \gV$ by the closedness of these sets, it follows that
 \[
 \Psi^k(\bar{\vy}, \bar{\vv}) \leq \Psi^k(\vy^*,\bar{\vv})\leq \Psi^k(\vy^*, \vv^*).
 \]
As discussed, the sequence $\{\Psi^k(\vy^{k, j}, \vv^{k, j})\}_{j=1}^\infty$ is monotonically decreasing and converges to a unique limit. Therefore, we must have $\Psi^k(\bar{\vy},\bar{\vv}) = \Psi^k(\vy^*,\vv^*)$, and consequently
\[
\Psi^k(\bar{\vy},\bar{\vv}) = \Psi^k(\vy^*,\bar{\vv})=\Psi^k(\vy^*,\vv^*).
\]
Following a similar argument as in the first part of the proof, we then obtain $\vv^* = \bar{\vv}$ and
\begin{eqnarray*}
&& \Psi^k(\vy^*,\vv^*) \leq \Psi^k(\vy^*,v),\quad \forall \vv\in \gV,\\
&& \Psi^k(\vy^*,\vv^*) \leq \Psi^k(y,\vv^*),\quad \forall \vy\in \Upsilon_\tau,
\end{eqnarray*}
which indicates that $(\vy^*,\vv^*)$ is a partially optimal solution of Problem $\text{PPLA}_k$. \hfill\QED

\subsection{Proof of Proposition \ref{prop:estimate}}
\label{app:estimate}

By Theorem \ref{thm:ama}-(i), it follows that 
\begin{equation}
    \Psi^k(\vy^{k+1}, \vv^{k+1}) < \Psi^k(\vy^{k,0}, \vv)=\Psi^k(\vy^k, \vv), \quad \forall \vv \in \gV.
    \label{t1-1}
\end{equation}
Setting $\vv$ be the unique $\vv^* \in \gV^*(\vy^k)$ in the above inequality, we then obtain
\begin{equation}
    \Psi^k(\vy^{k+1}, \vv^{k+1}) <  \Psi^k(\vy^k, \vv^*).
    \label{eq:p5-1}
\end{equation}
Noting that $\vv^* \in \gV^*(\vy^k)$ implies that $g(\vy^k) = f(\vv^*; \vy^k)$, we then have
\begin{equation*}
    \Phi(\vy^k, \vv^*; \vy^k, \vv^k) = f(\vv^*; \vy^k) - [g(\vy^k) + \nabla g(\vy^k)^{\rm T} \cdot (\vy^k - \vy^k)] = f(\vv^*; \vy^k) - g(\vy^k) = 0,
\end{equation*}
and accordingly, by the definition of $\Psi^k(\vy^k, \vv^*)$,
\begin{equation}
\begin{split}
    \Psi^k(\vy^k, \vv^*) &= F(\vy^k, \vv^*) + \rho^k \cdot \Phi(\vy^k, \vv^*; \vy^k, \vv^k) + \rho^k \cdot \beta^k \cdot \|(\vy^k - \vy^k, \vv^* - \vv^k) \|_2^2 \\
    &= F(\vy^k, \vv^*) + \rho^k \cdot \beta^k \cdot \|\vv^* - \vv^k \|_2^2
    \label{eq:p5-2}
\end{split}
\end{equation}
Combining Equation \eqref{eq:p5-1} and \eqref{eq:p5-2} then yields 
\begin{equation}
\begin{split}
    \Psi^k(\vy^{k+1}, \vv^{k+1}) < F(\vy^k, \vv^*) + \rho^k \cdot \beta^k \cdot \|\vv^* - \vv^k \|_2^2 \leq b_u + \theta_u \cdot c_u,\label{righthand}
\end{split}
\end{equation}
where the second inequality follows from the definitions of $b_u$ and $c_u$ and Condition \eqref{condition}.

Meanwhile, by the definitions of $\Psi^k(\vy^{k+1}, \vv^{k+1})$ and $b_u$, it can be obtained that
\begin{eqnarray*}\label{legthand}
        \Psi^k(\vy^{k+1}, \vv^{k+1}) \geq F(\vy^{k+1}, \vv^{k+1}) + \rho^k \cdot \Phi(\vy^{k+1}, \vv^{k+1}; \vy^k,\vv^k)\geq b_l + \rho^k \cdot \Phi(\vy^{k+1}, \vv^{k+1}; \vy^k,\vv^k),
\end{eqnarray*}
which, together with Equation \eqref{righthand}, yields
\[
    \Phi(\vy^{k+1}, \vv^{k+1}; \vy^k, \vv^k)\leq \frac{b_u + \theta_u \cdot c_u -b_l}{\rho^k}.
\]
Eventually, by Condition \eqref{condition}, we have $\Phi(\vy^{k+1}, \vv^{k+1}; \vy^k, \vv^k) \leq \epsilon$, which concludes the proof. 
\hfill\QED

\subsection{Proof of Proposition \ref{prop:decrease}}
\label{app:decrease}

First, note that Equation~\eqref{t1-1} applies in this setting. Setting $\vv = \vv^k$ then yields
\begin{equation}
\Psi^k(\vy^{k+1}, \vv^{k+1}) < \Psi^k(\vy^k, \vv^k).
\label{eq:p6-m0}
\end{equation}

(i) By the definition of $\Psi^k(\vy^{k+1}, \vv^{k+1})$, we have
\begin{equation}
\begin{split}
    \Psi^k(\vy^{k+1}, \vv^{k+1}) - \rho^k \cdot \beta^k \cdot &\|(\vy^{k+1} - \vy^k, \vv^{k+1} - \vv^k) \|_2^2 \\
    &= F(\vy^{k+1}, \vv^{k+1}) + \rho^k \cdot \Phi(\vy^{k+1}, \vv^{k+1}; \vy^k, \vv^k).
\end{split}
\label{eq:p6-1}
\end{equation}
Meanwhile, setting $\vy = \vy^{k + 1}$ and $\vv = \vv^{k + 1}$ in Equation \eqref{eq:phi-Phi} yields
\begin{equation}
\begin{split}
\varphi(\vv^{k + 1}; \vy^{k + 1}) &\leq \Phi(\vy^{k + 1}, \vv^{k + 1}; \vy^k, \vv^k).
\end{split}
\label{eq:p6-2}
\end{equation}
Combining Equations \eqref{eq:p6-1} and \eqref{eq:p6-2}, we then obtain that
\begin{equation}
\begin{split}
    \psi(\vy^{k+1}, \vv^{k+1}; \rho^k) &= F(\vy^{k+1}, \vv^{k+1}) + \rho^k \cdot \varphi(\vv^{k + 1}; \vy^{k + 1}) \\
    &\leq F(\vy^{k+1}, \vv^{k+1}) + \rho^k \cdot \Phi(\vy^{k + 1}, \vv^{k + 1}; \vy^k, \vv^k) \\
    &= \Psi^k(\vy^{k+1}, \vv^{k+1}) - \rho^k \cdot \beta^k \cdot \|(\vy^{k+1} - \vy^k, \vv^{k+1} - \vv^k) \|_2^2.
    \label{eq:p6-m1}
\end{split}
\end{equation}

(ii) By the definition of $\Psi^k(\vy^k, \vv^k)$, we have
\begin{equation}
    \Psi^k(\vy^k, \vv^k) = F(\vy^k, \vv^k) + \rho^k \cdot \Phi(\vy^k, \vv^k; \vy^k,\vv^k) + \rho^k \cdot \beta^k \cdot \|(\vy^k-\vy^k, \vv^k-\vv^k)\|_2^2.
    \label{eq:p6-3}
\end{equation}
Meanwhile, by the definition of $\Phi(\vy^k, \vv^k; \vy^k,\vv^k)$, we have
\begin{equation}
    \Phi(\vy^k, \vv^k; \vy^k,\vv^k) = f(\vv^k; \vy^k) -[g(\vy^k) + \nabla g(\vy^k)^{\rm T} \cdot (\vy^k-\vy^k)] = f(\vv^k; \vy^k) - g(\vy^k) = \varphi(\vv^k; \vy^k).
    \label{eq:p6-4}
\end{equation}
Combining Equations \eqref{eq:p6-3} and \eqref{eq:p6-4} then implies that
\begin{equation}
    \Psi^k(\vy^k, \vv^k) = F(\vy^k, \vv^k) + \rho^k \cdot \varphi(\vv^k; \vy^k) = \psi(\vy^{k}, \vv^{k}; \rho^k).
    \label{eq:p6-m2}
\end{equation}

Finally, Equation~\eqref{eq:suff-de} is proved by combining Equations~\eqref{eq:p6-m0}, \eqref{eq:p6-m1}, and \eqref{eq:p6-m2}. \hfill\QED

\subsection{Proof of Proposition \ref{prop-welldefined}}  
\label{app:prop-welldefined}

First, we claim that there exists $k_0 \in \sN$ such that $\Phi(\vy^{k+1}, \vv^{k+1}; \vy^k, \vv^k) \leq \epsilon_3$ for all $k \geq k_0$. Suppose, to the contrary, that there exists a subsequence $\gK \subseteq \sN$ such that $\Phi(\vy^{k+1}, \vv^{k+1}; \vy^k, \vv^k) > \epsilon_3$ for all $k \in \gK$. By the penalty parameter updating rule in Step~(iii), it then follows that $\rho^k \to \infty$ as $k \to \infty$. Consequently, there must exist $k_1 \in \mathbb{N}$ such that
\[
\rho_{k} \geq \frac{b_u + \theta_u \cdot c_u - b_l}{\epsilon_3}, \quad \forall k \geq k_1.
\]
Now, applying Proposition \ref{prop:estimate}, we obtain that
\[
\Phi(\vy^{k+1}, \vv^{k+1}; \vy^k, \vv^k)  \leq \epsilon_3, \quad \forall k \geq k_1,
\]
which contradicts the assumption that $\Phi(\vy^{k+1}, \vv^{k+1}; \vy^k, \vv^k) > \epsilon_3$ for all $k \in \gK$.

By the updating rule in Step~(iii), the penalty parameter $\rho^k$ is no longer increased for all $k \geq k_0$, that is, $\rho^k = \rho^{k_0}$ for all $k \geq k_0$. Then, by Proposition~\ref{prop:decrease}, it follows that for all $k \geq k_0$,
\begin{equation}
\begin{split}
\psi(\vy^{k + 1}, \vy^{k + 1}; \rho^{k^0}) &\leq  \psi(\vy^{k}, \vy^{k}; \rho^{k^0}) - \rho^{k} \cdot \beta^{k} \cdot \|(\vy^{k + 1} - \vy^k, \vv^{k + 1} - \vv^k)\|_2^2\\ 
&\leq \psi(\vy^{k}, \vy^{k}; \rho^{k^0}) - \theta_l \cdot \|(\vy^{k + 1} -\vy^k, \vv^{k + 1} - \vv^k)\|_2^2, \label{pr6-eq1}
\end{split}
\end{equation}
where the second inequality follows from the enforcement of $\rho_k \cdot \beta_k \geq \theta_l$ in Step~(iv) of Algorithm~\ref{alg:main}. This indicates that the sequence $\{\psi(\vy^k, \vv^k; \rho^{k_0})\}_{k \geq k_0}$ is monotonically decreasing and thus convergent to a unique limit. Taking limits as $k \to \infty$ on both sides of Equation~\eqref{pr6-eq1} then yields 
\begin{equation}\label{limit}
(\vy^{k + 1} - \vy^k, \vv^{k + 1} - \vv^k) \to 0.
\end{equation}
Consequently, there must exist an integer $k_1 \geq k_0$ such that
\[
\|\vy^{k + 1} - \vy^k\|_2 \leq \epsilon_1 \quad \text{and} \quad \|\vv^{k + 1} - \vv^k\|_2 \leq \epsilon_2, \quad \forall k\geq k_1.
\]
Therefore, once $k > k_1$, the criteria in Step~(ii) of Algorithm~\ref{alg:main} will be satisfied. \hfill\QED

\subsection{Proof of Theorem \ref{thm-main}}
\label{app:thm-main}

Noting from Proposition~\ref{defi1} that $\nabla g(\vy)$ is continuous and from Proposition~\ref{prop-cardiset} that $\Upsilon_\tau$ is compact, the Heine–Cantor theorem ensures that $\nabla g(\vy)$ is uniformly continuous on $\Upsilon_\tau$.
Hence, there exists $\delta > 0$ such that as long as $\|\vy^{k + 1} - \vy^{k}\|_2 < \delta$, then
\begin{equation}\label{proof:eq2}
\|\nabla g(\vy^{k + 1})-\nabla g(\vy^{k})\|_2 \leq \frac{\epsilon}{2\sqrt{2} \cdot \rho^{k_0}}.
\end{equation}
We now set
$
    \epsilon_1 = \min\left\{\delta, \frac{\epsilon}{4\sqrt{2} \cdot \theta_u}\right\}.
$
By Proposition~\ref{prop-welldefined}, there exists $k_0 \in \mathbb{N}$ such that Algorithm~\ref{alg:main} terminates at $k = k_0$ with
\begin{equation}\label{proof:eq1}
    \|\vy^{k_0 + 1} - \vy^{k_0}\|_2\leq \epsilon_1,\quad \|\vv^{k_0 + 1} - \vv^{k_0}\|_2\leq \epsilon_2 = \frac{\epsilon}{2\sqrt{2} \cdot \theta_u} , \quad \text{and} \quad \Phi(\vy^{k_0+1}, \vv^{k_0+1}; \vy^k_0, \vv^k_0) \leq \epsilon_3 = \epsilon,
\end{equation}
where the last inequality immediately yields
\begin{equation}
    \varphi(\vy^{k_0 + 1}, \vv^{k_0 + 1}) \leq \epsilon.
    \label{eq:eps-fea}
\end{equation}
Moreover, by the choice of $\epsilon_1$, inequality~\eqref{proof:eq2} is satisfied, and we also have
\begin{equation}
    \|\vy^{k_0 + 1} - \vy^{k_0}\|_2 \leq \frac{\epsilon}{4\sqrt{2} \cdot \theta_u}.
    \label{eq:y-gap}
\end{equation}

Since $(\vy^{k_0+1}, \vv^{k_0+1})$ is a partially optimal solution to PPLA$_{k_0}$, it must satisfy the stationarity conditions of PPLA$_{k_0}$ \citet[see, e.g.,][Theorem~10.1]{rockafellar2009}. By Definition \ref{defi1},
\begin{align*}
\vzero \in \nabla_{\vy} F(\vy^{k_0+1}, \vv^{k_0+1}) + \rho^{k_0} \cdot [\nabla f(\vv^{k_0+1}; \vy^{k_0+1}) -\nabla g(\vy^{k_0})+2 \cdot \beta^{k_0} \cdot (\vy^{k_0+1}-\vy^{k_0})] + \cN_{\Upsilon_\tau}(\vy^{{k_0}+1}),\\
\vzero \in \nabla_{\vv} F(\vy^{k_0+1}, \vv^{k_0+1}) + \rho^{k_0} \cdot [\nabla_v f(\vv^{k_0+1}; \vy^{k_0+1})+ 2 \cdot \beta^k \cdot (\vv^{k_0+1} - \vv^{k_0})] + \cN_{\gV}(\vv^{k_0+1}).
\end{align*}
Setting $\bm \zeta_1^{k_0} = -2 \cdot \rho^{k_0} \cdot \beta^{k_0} \cdot (\vy^{{k_0}+1} - \vy^{k_0})+ \rho^{k_0} \cdot [\nabla g(\vy^{k_0}) - \nabla g(\vy^{{k_0} + 1})]$ and $\bm \zeta_2^{k_0} = - 2 \cdot \rho^{k_0} \cdot \beta^{k_0} \cdot (\vv^{{k_0}+1}-\vv^{k_0})$, we then obtain
\begin{eqnarray}
&& \bm \zeta_1^{k_0} \in \nabla_{\vy} F(\vy^{{k_0}+1}, \vv^{{k_0}+1}) + \rho^{k_0} \cdot [\nabla f(\vv^{{k_0}+1}; \vy^{{k_0}+1}) -\nabla g(\vy^{{k_0}+1})] + \cN_{\Upsilon_\tau}(\vy^{{k_0}+1}),\label{proof:eq4}\\
&& \bm \zeta_2^{k_0} \in \nabla_{\vv} F(\vy^{{k_0}+1}, \vv^{{k_0}+1}) + \rho^{k_0} \cdot \nabla_{\vv} f(\vv^{{k_0}+1}; \vy^{{k_0}+1}) + \cN_{\gV}.(\vv^{k_0+1}).\label{proof:eq5}
\end{eqnarray}
It can then be checked that
\begin{equation}
\begin{split}
\|\bm \zeta_1^{k_0}\|_2 &\leq  2 \cdot \rho^{k_0} \cdot \|\beta^{k_0} \cdot (\vy^{{k_0}+1}-\vy^{k_0})\|_2+ \rho^{k_0} \cdot \|\nabla g(\vy^{k_0})-\nabla g(\vy^{{k_0}+1})\|_2\\
&\leq  2 \cdot \theta_u \cdot \|\vy^{{k_0}+1}-\vy^{k_0}\|_2 + \rho^{k_0} \cdot \|\nabla g(\vy^{k_0}) - \nabla g(\vy^{{k_0}+1})\|_2 \\
&\leq  2\cdot \theta_u \cdot \frac{\epsilon}{4\sqrt{2} \cdot \theta_u} + \rho^{{k_0}} \cdot \frac{\epsilon}{2\sqrt{2} \cdot \rho^{{k_0}}} = \frac{\epsilon}{\sqrt{2}},
\label{eq:zeta_1}
\end{split}
\end{equation}
where the second inequality follows from the enforcement of $\rho^k \cdot \beta^k \leq \theta_u$ throughout Algorithm~\ref{alg:main}, and the third inequality uses Equations~ \eqref{proof:eq2} and \eqref{eq:y-gap}. Similarly, we can also obtain that
\begin{equation}
\|\bm \zeta_2^{k_0}\|_2= 2 \cdot \rho^{k_0} \cdot \beta^{k_0} \cdot  \|\vv^{{k_0}+1}-\vv^{k_0}\|\leq 2 \cdot \theta_u \cdot \frac{\epsilon}{2\sqrt{2} \cdot \theta_u}=\frac{\epsilon}{\sqrt{2}}.
\label{eq:zeta_2}
\end{equation}
Equations \eqref{eq:zeta_1} and \eqref{eq:zeta_2} combined together then ensure that
\begin{equation}
    \|(\bm \zeta_1^{k_0}, \bm \zeta_2^{k_0})\|_2 \leq \epsilon.
    \label{eq:zeta}
\end{equation}
Together with \eqref{eq:eps-fea}, \eqref{proof:eq4}, \eqref{proof:eq5}, and the fact that $\vy^{k} \in \Upsilon_\tau$ and $\vv^{k} \in \gV$ throughout Algorithm~\ref{alg:main}, this verifies that $(\vy^{{k_0}+1}, \vv^{{k_0}+1})$ is an $\epsilon$-approximate stationary point of CCBCEP-E~\eqref{equi-p}. \hfill\QED

\end{document}